\documentclass{amsart}

\usepackage{amsmath,amssymb}
\usepackage{amsthm}
\usepackage[singlelinecheck=off]{caption}
\usepackage{microtype}
\usepackage{graphicx}
\usepackage{algorithm,algpseudocode}
\usepackage{mathrsfs}
\usepackage{epstopdf}
\usepackage{float}
\usepackage{url}
\usepackage{fancyhdr}
\usepackage{subfig}
\usepackage{epstopdf}

\makeatletter
\def\BState{\State\hskip-\ALG@thistlm}
\makeatother

\usepackage{epstopdf}
\usepackage{nicefrac}
\newcommand{\Tr}{\mathrm{Tr}}

\newtheorem{definition}{Definition}
\newtheorem{theorem}{Theorem}
\newtheorem{corollary}{Corollary}
\newtheorem{proposition}{Proposition}

\DeclareMathOperator{\diag}{diag}
\DeclareMathOperator{\conv}{conv}

\title{Convex relaxation approaches for strictly correlated density functional theory}
\author{Yuehaw Khoo
	\and Lexing Ying }

\address{
  Department of Mathematics,
  Stanford University,
  Stanford, CA 94305
}

\begin{document}



\maketitle

\begin{abstract}
	In this paper, we introduce methods from convex optimization to solve the multimarginal transport
	type problems arise in the context of density functional theory. Convex relaxations are used to
	provide outer approximation to the set of $N$-representable 2-marginals and 3-marginals, which in
	turn provide lower bounds to the energy. We further propose rounding schemes to obtain upper bound
	to the energy. Numerical experiments demonstrate a gap of the order of $10^{-3}$ to $10^{-2}$
	between the upper and lower bounds. The Kantorovich potential of the multi-marginal transport
	problem is also approximated with a similar accuracy.
\end{abstract}

\section{Introduction}
\label{section: introduction} 
We propose a novel convex relaxation framework for solving multimarginal-transport type problems, in
the context of density functional theory for strictly correlated electrons. More precisely, we
consider the type problems that takes the form
\begin{gather}
\label{general MM}
\underset{\lambda_1,\ldots,\lambda_N,\mu \in \Pi(\lambda_1,\ldots,\lambda_N)}{\inf} \sum_{i=1}^N
g_i(\lambda_i) + \int_{X_1\times\cdots \times X_N} f(x_1,\ldots,x_N) d\mu(x_1,\ldots,x_N),\quad
\mathcal{A}_i(\lambda_i) = b_i,\ i=1,\ldots,N
\end{gather}
where $g_i(\cdot),i=1,\ldots,N$ are convex functionals, $\mathcal{A}_i, i=1,\ldots,N$ are some
linear operators, $ \Pi(\lambda_1,\ldots,\lambda_N)$ denotes the space of probability measures on
$X_1\times\cdots\times X_N$ with marginals $\lambda_1,\ldots,\lambda_N$. In this paper, the domain
of the cost $X_1\times \cdots \times X_N$ is discrete and the cost function $f$ has the form
\begin{equation}
\label{general cost}
f(x_1,\ldots,x_N) =  \sum_{i,j=1, i> j}^N C_{ij}(x_i,x_j).
\end{equation}
A particular situation that we are interested in is when $f(x_1,\ldots,x_N)$ and
$\mu(x_1,\ldots,x_N)$ are symmetric when any $x_i$ and $x_j$ are swapped, i.e., $g_i := g$, and
$C_{ij} := C$ for $i,j = 1,\ldots,N$. In such a situation, the task is to solve
\begin{equation}
\label{symmetric MM}
\underset{\lambda, \mu \in \Pi_{N,\text{sym}}(\lambda)}{\inf} g(\lambda) + \int_{X^N}
f(x_1,\ldots,x_N) d\mu(x_1,\ldots,x_N),\quad \mathcal{A}(\lambda) = b
\end{equation}
where $ \Pi_{N,\text{sym}}(\lambda)$ denotes the space of symmetric probability measures on $X^N$
with the marginals being $\lambda$. Solving this problem is particularly useful in the context of
density functional theory (DFT), where the density for many-electrons is indeed symmetric. A brief
introduction to how such a problem can arise in DFT when the electrons are strictly correlated is
given in Section \ref{section: DFT}. Although Problem \eqref{symmetric MM} is a linear programming
problem for discrete $X$, the domain of optimization is exponentially large for any practical
computation.



\subsection{Background on DFT for strictly correlated electrons}\label{section: DFT}
A key task in density functional theory is to determine the minimum of an energy functional
$E(\rho)$ of the 1-marginal
\begin{equation}
\rho(x) = \int_{X^{N-1}} \vert \psi(x,\ldots,x_N)\vert^2 dx_2 dx_3\ldots dx_N,
\end{equation}
where $\psi(x_1,\ldots,x_N)$ is a many-body wavefunction for $N$ electrons (due to the properties
of electrons $\vert \psi(x_1,\ldots,x_N) \vert^2$ is symmetric). In this paper, we consider an
energy functional introduced in \cite{gori2009density}
\begin{equation}
  \label{total energy}
  E(\rho) =  V_\text{ee}^\text{SCE}(\rho) + E_\text{kd}(\rho) + \int_{X} v_\text{ext}(x) \rho(x)dx,
\end{equation}
which is suitable for studying strongly correlated electrons. The functional $E_\text{kd}(\rho)$
corresponds to kinetic energy with some correction terms, $v_\text{ext}$ is some external potential
(for example potential exerted by nuclei), and the central object of the study is the strictly
correlated density functional $V_\text{ee}^\text{SCE}(\rho)$ defined as
\begin{equation}
  \label{SCE OT}
  V_\text{ee}^\text{SCE}(\rho):=\inf_{\lambda, \mu\in \Pi_{\text{N,sym}}(\lambda)} \int_{X^2}
  \sum_{i,j=1,i> j}^N \frac{1}{\| x_i - x_j \|} d\mu(x_1,\ldots,x_N),\quad \lambda = \rho.
\end{equation}
This framework of DFT gives rise to two following problems:
\begin{itemize}
\item Solving for the strictly correlated density functional $V^\text{SCE}_\text{ee}(\rho)$ via the
  optimization problem \eqref{SCE OT}. This is in fact the well known multi-marginal optimal
  transport problem.\newline
\item Direct minimization of the total energy functional $E(\cdot)$ in \eqref{total energy}, when
  the kinetic energy $E_\text{kd}(\rho)$ is either convex or negligible (thus can be dropped). In
  this case, the minimization problem takes the form
  \begin{eqnarray}
	&\ &\underset{\rho}{\inf}\quad V_\text{ee}^\text{SCE}(\rho) + E_\text{kd}(\rho) + \int_X v_\text{ext}(x)\rho(x) dx\cr
	&\Leftrightarrow & \underset{\rho}{\inf}\quad  E_\text{kd}(\rho) + \int_X v_\text{ext}(x)\rho(x) dx + \inf_{\lambda, \mu\in \Pi_{\text{N,sym}}(\lambda),\lambda=\rho} \int_{X^2} \sum_{i,j=1,i> j}^N \frac{1}{\| x_i - x_j \|} d\mu(x_1,\ldots,x_N)  \label{indirect total minimization}\\
	&\Leftrightarrow & \underset{\lambda, \mu\in \Pi_{\text{N,sym}}(\lambda)}\inf E_\text{kd}(\lambda) + \int_X v_\text{ext}(x)\lambda(x) dx +  \int_{X^2} \frac{1}{\| x_i - x_j \|} d\mu(x_1,\ldots,x_N)	\label{total minimization}
	\end{eqnarray}
\end{itemize}
Notice that the first problem, i.e. \eqref{SCE OT}, takes the form of \eqref{symmetric MM} when
$\mathcal{A}$ is the identity and $b=\rho$, while the second problem, presented in \eqref{total
  minimization}, takes the form of \eqref{symmetric MM} when the constraint $\mathcal{A}(\lambda)=b$
is absent.

\subsection{Our contributions}
In this paper, we work with an equivalent formulation of \eqref{symmetric MM} in terms of the
2-marginals. Although this seems to break the aforementioned complexity barrier for solving
\eqref{symmetric MM}, enforcing that the 2-marginals being the marginalization of a probability
measure on $X^N$, is non-trivial. Leveraging the results of \cite{friesecke2018breaking} concerning
the extreme points of the $N$-representable 2-marginals, we propose a semidefinite programming (SDP)
relaxation, SDP-Coulomb, to provide an outer approximation to the set of $N$-representable
2-marginals, therefore breaking the complexity barrier in optimizing the high-dimensional measure in
\eqref{symmetric MM}. The property of the proposed SDP is discussed in light of the results in
\cite{friesecke2018breaking}. We further propose a tighter convex relaxation SDP-Coulomb2 based on a
formulation of \eqref{symmetric MM} in terms of the 3-marginals. As the proposed convex relaxations
only provide lower bounds to the energy, we further propose rounding schemes to give upper
bounds. Numerical simulations show that the proposed approaches give a relative gap between the
upper and lower bounds of size $10^{-3}$ to $10^{-2}$, which in turn sets an upper bound on the
approximation accuracy. Before delving into the details, in Fig. \ref{figure: largescale}, we show
an example where we solve the multi-marginal transport problem \eqref{SCE OT} with $N=8$,
$\rho(x)\propto \exp(-x^2/\sqrt{\pi})$, and the discrete domain $X$ has size $\vert X \vert =
1600$. The running time is 2560s. Such problem size would be impossible to be solved by traditional
methods such as linear programming since it requires the storage of a tensor with $10^{25}$
entries. Moreover, in this example, we obtain an estimate of ${ V}^\text{SCE}_\text{ee}(\rho)$ with
3.6e-04 error.
\begin{figure}[!ht]
  \centering
  \includegraphics[width=0.47\columnwidth]{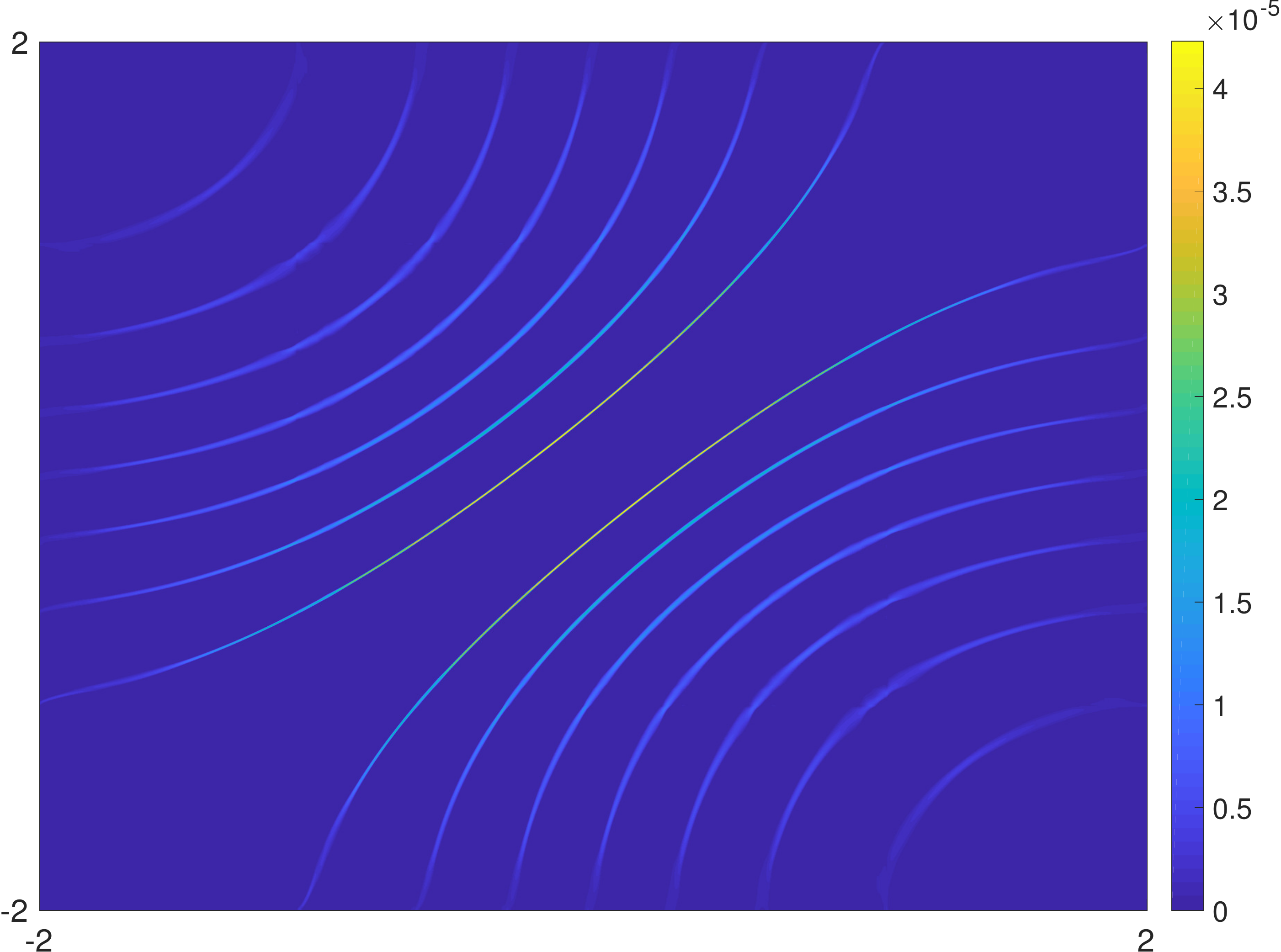}
  \caption{Approximate solution using SDP-Coulomb to the multimarginal transport problem \eqref{SCE
      OT} with $\rho(x)\propto \exp(-x^2/\sqrt{\pi})$ being the marginal. Here $N=8$, $\vert X \vert
    = 1600$, $d=1$. The error of the energy is 3.6e-04.}\label{figure: largescale}
\end{figure}

\subsection{Prior works}
The consideration of numerically solving an optimal transport problem with a Coulomb cost is a relatively new field. In
\cite{mendl2013kantorovich}, the dual problem to Problem \eqref{symmetric MM} is solved, via a
parameterization of the dual function. In \cite{chen2014numerical}, linear programming is applied to
solve the problem involving 2-electrons in 3D as part of a self-consistent DFT iterations. In
\cite{benamou2016numerical}, Sinkhorn scaling algorithm is applied to an entropic regularized
problem of \eqref{symmetric MM}. Although these methods have shown various levels of success in
practice, the constraints or variables involved grow exponentially in the number of electrons.

\subsection{Organization}
In Section \ref{section: method}, we detail the proposed SDP relaxation for Problem \eqref{symmetric
  MM} in terms of the 2-marginal. In Section \ref{section: properties}, we characterize the property
of the SDP relaxation. In Section \ref{section: 3mar}, a further tightening of the SDP relaxation is
proposed by formulating Problem \eqref{symmetric MM} in terms of the 3-marginal. In Section
\ref{section: implementation}, rounding schemes are provided to obtain an upper bound of the
energy. In Section \ref{section: numerical}, we demonstrate the effectiveness of the proposed method
through numerical examples.

\subsection{Notation}

In what follows, $I$ is used to denote the identity matrix as usual and we use $A^T$ to denote the matrix
transpose. For a $p$-dimensional tensor $T$, $T(j_1,j_2,\ldots,j_p)$ denotes its
$(i_1,\ldots,i_p)$-th entry. MATLAB notation ``$:$'' is used to extract a slice of a tensor. For
example for a matrix $A$, $A(:,i)$ gives the $i$-th column of the matrix. $\mathbf{1}$ is used to
denote an all-one vectors of appropriate size. For a matrix $A \in \mathbb{R}^{L\times L}$, the
operator $\diag(A)\in \mathbb{R}^L$ extracts the diagonal of $A$ and $\diag^*$ denotes the adjoint
of $\diag$. The notations $\odot$ and $\otimes$ are used to denote the Hadamard and tensor products
respectively. For a $p$-dimensional tensor $T$, $\| T \|_F^2$ is defined as
\begin{equation}
  \| T \|_F^2 := \sum_{i_1,\ldots,i_p} \vert T(i_1,\ldots,i_p) \vert^2
\end{equation}

\section{Proposed method}
\label{section: method}
In this section, we proposed an SDP relaxation to solve the equivalent problem of \eqref{symmetric
  MM} in terms of the 2-marginals. In terms of the 2-marginals $\gamma_{ij}$, the cost of
\eqref{symmetric MM} is
\begin{equation}
\label{2mar MM con}
g(\lambda) + \sum_{i,j=1, i<j}^N \int_{X^N} C_{ij}(x_i,x_j) d\gamma_{ij}(x_i,x_j) = g(\lambda) +
\frac{N(N-1)}{2} \int_{X^2} C(x,y) d\gamma(x,y),
\end{equation}
where $\gamma_{ij}(x_i,x_j) = \gamma(x_i,x_j)$ due to the symmetry of $\mu$. The 2-marginal $\gamma$
is called an $N$-representable measure, since it comes from the marginalization of a symmetric
probability measure on $X^N$. A more general definition for $k$-marginal is given below.
\begin{definition}
  A $k$-marginal on $X^N$ is called $N$-representable if it results from the marginalization of a
  symmetric probability distribution on $X^N$.
\end{definition}
As we consider a discrete state space $X$, Problem \eqref{2mar MM con} in terms of the discrete
2-marginals takes the form
\begin{eqnarray}
\label{2mar MM}
\min_{\gamma\in \mathbb{R}^{L\times L}} &\ & g(\gamma\mathbf{1})+  \frac{N(N-1)}{2} \Tr\big(C \gamma)  \\
\text{s.t.}&\ & \gamma\ \text{is $N$-representable}\cr
&\ & \diag(\gamma) = 0\cr
&\ & (\gamma \mathbf{1}) = b.\nonumber
\end{eqnarray}
Here we added a problem-dependent constraint $\diag(\gamma) = 0$, due to the fact the Coulomb cost
$C(\cdot,\cdot)$ is infinity when two arguments coincide.
To derive an SDP relaxation to \eqref{2mar MM}, one first needs a characterization of the
$N$-representable 2-marginals. For this, we leverage the following result from
\cite{friesecke2018breaking}, where $\conv(S)$ denotes the convex hull of a set $S$.
\begin{theorem}
  The set of discrete $N$-representable 2-marginals is  $\conv(\Gamma_2)$ where
  \begin{equation}
	\Gamma_2 = \left\{ \frac{N}{N-1} \lambda \lambda^T - \frac{1}{N-1} \diag^*(\lambda) \ \vert\ \lambda
    \in \{0,\nicefrac{1}{N},\nicefrac{2}{N},\ldots,1\}^{\vert X \vert },\quad \lambda^T \mathbf{1} =
    1 \right\}.
  \end{equation}
  Moreover $\Gamma_2$ is the set of extreme points of $\conv(\Gamma_2)$.
\end{theorem}
Since we are interested in the 2-marginals $\gamma$ where the diagonal element is zero, we
characterize the subset $\tilde \Gamma_2\subset \Gamma_2$ with this extra zero constraint in the
following corollary. Let
\begin{equation}
\label{binary dist}
\mathcal{B}_N(X) = \{\lambda\in \mathbb{R}^{\vert X \vert}\ \vert\ \lambda^T\mathbf{1} =
1,\ \lambda(i) \in \{0,1/N\},\ i=1,\ldots,\vert X \vert \}
\end{equation}
which denotes the set of binarized probability vector on a discrete domain $X$. 
\begin{corollary}
  \label{corollary: extreme point}
  Let
  \begin{equation}
	\tilde \Gamma_2 = \left\{ \frac{N}{N-1} \lambda \lambda^T - \frac{1}{N-1} \diag^*(\lambda)
    \ \vert\ \lambda \in \mathcal{B}_N(X) \right\},
  \end{equation}
  then 
  \begin{equation}
	\label{extreme point}
	\conv(\tilde \Gamma_2) = \{\gamma\in \mathbb{R}^{\vert X\vert\times \vert X \vert} \ \vert
    \ \gamma\ \text{is $N$-representable},\ \diag(\gamma) = 0 \}.
  \end{equation}
  Moreover, $\tilde \Gamma_2$ is the extreme points of $\conv(\tilde \Gamma_2)$.
\end{corollary}
For completeness, a short proof of Corollary \ref{corollary: extreme point} is presented in Section
\ref{section: properties}. With this characterization, an equivalent formulation of \eqref{2mar MM}
is obtained as
\begin{eqnarray}
  \min_{\gamma\in \mathbb{R}^{\vert X \vert \times \vert X \vert }} &\ & g(\gamma\mathbf{1})+  \frac{N(N-1)}{2} \Tr\big(C \gamma)  \\
  \text{s.t.} &\ & \gamma\in \conv(\tilde\Gamma_2)\cr
  &\ & \mathcal{A}(\gamma \mathbf{1}) = b.\nonumber
\end{eqnarray}
We claim that this is also equivalent to the following minimization problem:
\begin{eqnarray}
  \label{2mar MM2}
  \min_{\gamma,\lambda,a} &\ & g(\gamma \mathbf{1})+  \frac{N(N-1)}{2} \Tr(C \gamma)  \quad  \\
  \text{s.t.} &\ & \gamma = \frac{N}{N-1} \sum_{i=1}^m a(i) \lambda_i \lambda_i^T- \frac{1}{N-1} \diag^*(\sum_{i=1}^m a(i)\lambda_i)\cr
  &\ & \sum_{i=1}^m a(i) =1,\ a(i)\geq 0,\ i=1,\ldots,m\cr
  &\ & \lambda_i^T \mathbf{1} = 1,\ i=1,\ldots,m\cr
  &\ & \lambda_i\in \{0,1/N\}^{\vert X \vert  },\ i=1,\ldots,m\cr
  &\ & \mathcal{A}(\gamma \mathbf{1}) = b.\nonumber
\end{eqnarray}
Here, the first four constraints are equivalent to $\gamma$ being an element in $\conv(\tilde
\Gamma_2)$. The integer $m$ specifies the number of elements in $\tilde \Gamma_2$ needed for
representing $\gamma$, which depends on the number of linear constraints $\mathcal{A}(\gamma
\mathbf{1}) = b$. For the purpose of this section it is not
important to know what $m$ is, and we can just treat it as an arbitrary integer. A detail discussion on what $m$ is for the problem considered is provided in Section \ref{section: rounding 2}.


\subsection{Convex relaxation}
Problem \eqref{2mar MM2} involves optimizating over the set $\mathcal{B}_N(X)$, which has a
combinatorial complexity in the worst case. To cope with this issue, we propose the following convex
relaxation to Problem \eqref{2mar MM} and Problem \eqref{2mar MM2}:
\begin{eqnarray}
  \label{2mar SDPCoulomb}
  \min_{\gamma,\Lambda \in \mathbb{R}^{\vert X \vert \times \vert X \vert }} &\ & g(\gamma \mathbf{1})+
  \frac{N(N-1)}{2} \Tr(C \gamma)  \quad \text{(SDP-Coulomb)} \\
  \text{s.t.} &\ & \gamma = \frac{N}{N-1} \Lambda - \frac{1}{N-1} \diag^*(\Lambda \mathbf{1})\cr
  &\ & \mathcal{A}(\Lambda \mathbf{1}) = b\cr
  &\ &\Lambda \succeq 0\cr
  &\ &\Lambda \geq 0 \cr
  &\ & \mathbf{1}^T \Lambda \mathbf{1} = 1\cr
  &\ & \diag(\Lambda) =\frac{1}{N} \Lambda \mathbf{1}.\nonumber
\end{eqnarray}
The details of going from \eqref{2mar MM2} to \eqref{2mar SDPCoulomb} are presented in the
subsequent sections.

\subsubsection{Changing the variables to $\Lambda$}
We start to derive SDP-Coulomb from Problem \eqref{2mar MM2}. Instead of working with both sets of
variables $\{\lambda_i\}_{i=1}^{m}$ and $a$ as in Problem \eqref{2mar MM2}, we will only work with a
single matrix variable $\Lambda$. First let
\begin{equation}
  \label{Lambda def}
  \Lambda := \sum_{i=1}^m a(i) \lambda_i \lambda_i^T,\quad
  \lambda_i\in\mathcal{B}_N(X),\ i=1,\ldots,m.
\end{equation}
Since
\begin{equation}
\Lambda \mathbf{1} =  \sum_{i=1}^m a(i) \lambda_i \lambda_i^T \mathbf{1} = \sum_{i=1}^m a(i) \lambda_i,
\end{equation}
in terms of $\Lambda$ the 2-marginal $\gamma$ in \eqref{2mar MM2} becomes
\begin{equation}
\gamma = \frac{N}{N-1} \Lambda - \frac{1}{N-1} \diag^*(\Lambda \mathbf{1}).
\end{equation}
Notice that with such a change of variable, 
\begin{equation}
\gamma \mathbf{1} = \Lambda \mathbf{1}.
\end{equation}
\subsubsection{Constraints on $\Lambda$}
The variable $\Lambda$ defined in \eqref{Lambda def} belongs to a non-convex set as it is a quadratic form of the binarized vectors $\lambda_1,\ldots,\lambda_m$. In order to obtain the convex program SDP-Coulomb, we only enforce certain necessary conditions of  $\Lambda$ having the form in \eqref{Lambda def}. First
\begin{equation}
\label{C1}
\Lambda\succeq 0
\end{equation}
due to the fact that $a\geq 0$ in \eqref{Lambda def}. Then
\begin{equation}
\label{C2}
\Lambda \geq 0 
\end{equation}
since $a,\lambda_1,\ldots,\lambda_m \geq 0$ in \eqref{Lambda def}. Since $\sum_{i=1}^m
a(i)=1,\lambda_i^T \mathbf{1} = 1, i=1,\ldots,m$,
\begin{equation}
  \label{C3}
  \mathbf{1}^T \Lambda \mathbf{1}= 1.
\end{equation}
As each $\lambda_i \in B_N(X)$, therefore
\begin{equation}
  \lambda_i \odot \lambda_i  = \lambda_i/N,\quad i=1,\ldots,m
\end{equation}
implying
\begin{equation}
  \label{C4}
\diag(\Lambda) = \frac{1}{N} \Lambda \mathbf{1}.
\end{equation}
Together, the constraints \eqref{C1}, \eqref{C2}, \eqref{C3} and \eqref{C4} give the last four
constraints in SDP-Coulomb.

\subsection{Duality and the Kantorovich potential}
In \cite{mendl2013kantorovich}, the dual problem to \eqref{SCE OT}:
\begin{eqnarray}
\label{vsce}
V^\text{SCE}_\text{ee}(\rho) = \min_{v \in \mathbb{R}^{\vert X \vert }} &\ & v^T \rho \\
\text{s.t.} 
&\ & \sum_{\substack{k,l=1 \\ k\leq l}}^N C(i_k,i_l) - \frac{1}{N}\sum_{k=1}^N v(i_k) \geq 0,\ \forall (i_1,\ldots,i_N)\cr
\end{eqnarray}
is used to solve for $V^\text{SCE}_\text{ee}(\rho)$. This is called the Kantorovich problem and the dual variable $v$ is called the Kantorovich potential. Although the size of the optimization variable is reduced to $\vert X \vert$ when comparing to \eqref{SCE OT}, the dual formulation has number of constraints being exponential in $N$. We can also use SDP-Coulomb to provide an approximation to the Kantorovich potential. Let
\begin{equation}
\mathcal{A}(\Lambda \mathbf{1})=b \rightarrow \Lambda\mathbf{1} = \rho
\end{equation}
and $g=0$ in the cost, we have
\begin{eqnarray}
\label{vsce approx}
\hat V^\text{SCE}_\text{ee}(\rho) := \min_{\Lambda \in \mathbb{R}^{\vert X \vert \times \vert X \vert }} &\ & \frac{N^2}{2} \Tr[\left(C- \text{diag}^*(\text{diag}(C))\right) \Lambda] \\
\text{s.t.} 
&\ & w:\Lambda \mathbf{1} = \rho\cr
&\ & Y:\Lambda \succeq 0\cr
&\ & Z:\Lambda \geq 0 \cr
&\ & u:\diag(\Lambda) =\frac{1}{N} \Lambda \mathbf{1}\nonumber,
\end{eqnarray}
where the variables in front of the colon are the dual variables corresponding to the constraints. $\hat V^\text{SCE}_\text{ee}(\rho)$ can be seen as an approximation to $V^\text{SCE}_\text{ee}(\rho)$ in \eqref{vsce}. The dual to \eqref{vsce approx} is then
\begin{eqnarray}
\label{vsce dual}
\hat V^\text{SCE}_\text{ee}(\rho) = \max_{\substack{w \in \mathbb{R}^{\vert X \vert }, \\ Y,Z \in \mathbb{R}^{\vert X \vert \times \vert X \vert }}} &\ &  w^T \rho\\
\text{s.t.} 
&\ & \frac{N^2}{2} [C- \text{diag}^*(\text{diag}(C))] - \frac{1}{2}(\mathbf{1}w^T+ w \mathbf{1}^T)\cr &\ &\quad-\text{diag}^*(u)+\frac{1}{2N}(\mathbf{1}u^T+ u \mathbf{1}^T) = Y+Z, \cr
&\ & Y\succeq 0,\ Z\geq 0.\cr
\end{eqnarray}
The dual variable $w$ can be seen as an approximation to the Kantorovich potential $v$ in $\eqref{vsce}$. As pointed out in the literatures of DFT \cite{malet2012strong, mendl2013kantorovich, chen2014numerical}, the Kantorovich potential allows the functional derivative of $V^\text{SCE}_\text{ee}(\cdot)$ to be taken. From \eqref{vsce dual}, we make the following identification:
\begin{equation}
\frac{d V^\text{SCE}_\text{ee}(\rho) }{d \rho} \approx \frac{d \hat V^\text{SCE}_\text{ee}(\rho) }{d \rho} = w^*,
\end{equation}
where $w^*$ is the optimizer of \eqref{vsce dual}. The equality follows from the fact that for 
\begin{equation}
g(x) = \sup_{\alpha\in \Omega} f_\alpha(x),\quad \alpha^* =\text{argsup}_{\alpha\in \Omega} f_\alpha(x),
\end{equation}
where $ f_\alpha(x), \alpha\in \Omega$ are convex functions, a subgradient of $f_{\alpha^*}(x)$ is a subgradient of $g(x)$ \cite{boyd2004convex}. Obtaining the approximate functional derivative of $\hat V^\text{SCE}_\text{ee}(\cdot)$ can provide a mean to optimize \eqref{indirect total minimization} via self-consistent field iterations (for example in \cite{chen2014numerical}), when the dependence of $E_{kd}(\cdot)$ on $\rho$ is not analytically given. 
\section{Properties of SDP-Coulomb}
\label{section: properties}
The convex program SDP-Coulomb in Section \ref{section: method} intends to provide an outer approximation
to the 2-marginals. In this section, we show that the extreme points of the $N$-representable
2-marginals are contained in the set of the extreme points of the domain of SDP-Coulomb. We first
give the proof of Corollary \ref{corollary: extreme point}:
\begin{proof}
  It is clear in \eqref{extreme point} that the left hand side belongs to the right hand side. Now
  if $\gamma$ is $N$-representable, then
  \begin{equation}
	\gamma = 	\sum_{i=1}^m a(i) \bigg(\frac{N}{N-1}\lambda_i \lambda_i^T - \frac{1}{N-1}\diag^*(\lambda_i )\bigg), \ a \geq 0,\ a^T \mathbf{1} = 1,\ \lambda_i\in \{0,1/N\ldots,N/N\}^{\vert X \vert}.
	\end{equation}
  for $a\in \mathbb{R}^m$. The constraint $\diag(\gamma) = 0$ gives
  \begin{equation}
	\label{sum = 0}
	\sum_{i=1}^m a(i) (N \lambda_i\odot \lambda_i - \lambda_i )= 0
  \end{equation}
  where $\odot$ denotes pointwise product. Due to the domain of $\lambda_i$, $N \lambda_i\odot
  \lambda_i - \lambda_i \geq 0$. Then together with $a(i)\geq0$, the equation \eqref{sum = 0}
  implies $a(i)=0$ or $N \lambda_i\odot \lambda_i = \lambda_i$ for each $i$. This shows that
  $\lambda_i \in \{0,1/N\}^{\vert X \vert}$, implying in \eqref{extreme point} the right hand side
  belongs to the left hand side. Finally, it is clear that $\tilde \Gamma_2$ is the set of extreme
  points of $\conv(\tilde \Gamma_2)$, since $\tilde \Gamma_2$ is a subset of the extreme points
  $\conv(\Gamma_2)$ and $\conv(\tilde \Gamma_2)\subseteq \conv(\Gamma_2)$.
\end{proof}

In the following theorem, we show that $\tilde \Gamma_2$ also belongs to the set of the extreme points
for the feasible set of $\gamma$ used in Problem SDP-Coulomb in \eqref{2mar SDPCoulomb}, when the
constraint $\mathcal{A}(\Lambda \mathbf{1}) = b$ is absent. This shows that our convex relaxation is
rather tight.
\begin{theorem}
	\label{theorem: extreme}
  $\tilde \Gamma_2$ is a subset of the extreme points of the domain
	\begin{equation}
	  \label{SDP domain}
	  D = \left\{ \frac{N}{N-1} \Lambda - \frac{1}{N-1} \diag^*(\Lambda \mathbf{1})  \ \bigg\vert\ \Lambda\succeq 0,\ \Lambda\geq 0,\ \mathbf{1}^T \Lambda \mathbf{1}= 1,\ \diag(\Lambda) = \frac{1}{N} \Lambda \mathbf{1}  \right\},
	\end{equation}
    which is the feasible set of $\gamma$ in \eqref{2mar SDPCoulomb} when the constraint
    $\mathcal{A}(\Lambda \mathbf{1}) = b$ is absent.
\end{theorem}

\begin{proof}
	First $\tilde \Gamma_2$ is a subset of  $D$. We further need to show that each 
	\begin{equation}
	\label{gamma ext}
	\gamma_\text{ext} = \frac{N}{N-1} \lambda_\text{ext} \lambda_\text{ext}^T- \frac{1}{N-1} \diag^*(\lambda_\text{ext}), \quad \lambda_\text{ext}\in\mathcal{B}_N(X)
	\end{equation}
	in $\tilde \Gamma_2$ is also an extreme point in $D$. To this end, we simply show for every $\gamma_\text{ext}$, there exists some cost $B$ such that the unique maximizer to 
	\begin{equation}
	\label{extreme pt max}
	\max_\gamma \Tr(B\gamma),\quad \text{s.t.}\ \gamma \in D
	\end{equation}
	is $\gamma_\text{ext}$. If $\gamma_\text{ext}$ is the unique maximizer to \eqref{extreme pt max}, then $\gamma_\text{ext}\neq \sum_i a(i) \gamma_i$, where $\forall i\ \gamma_i\in D, a(i)> 0$, and $\sum_i a(i) = 1$. Otherwise, $\Tr(B\gamma_\text{ext}) = \sum_i a(i)\Tr(B \gamma_i)<\sum_i a(i) \Tr(B\gamma_\text{ext}) = \Tr(B\gamma_\text{ext})$ where the inequality is due to the fact that $\gamma_\text{ext}$ uniquely minimizes $\Tr(B\gamma)$. Let
	\begin{equation}
	B := \lambda_\text{ext} \lambda_\text{ext} ^T+ \frac{1}{N-1}\mathbf{1}\diag( \lambda_\text{ext} \lambda_\text{ext} ^T)^T.
	\end{equation}
	Then
	\begin{eqnarray}
	\Tr(B\gamma) &=& \Tr( \lambda_\text{ext} \lambda_\text{ext} ^T\gamma ) +  \Tr(\frac{1}{N-1}\mathbf{1}\diag( \lambda_\text{ext} \lambda_\text{ext} ^T)^T\gamma)\cr
	&=& \Tr( \lambda_\text{ext} \lambda_\text{ext} ^T\gamma ) + \Tr\bigg(\frac{1}{N-1}  \lambda_\text{ext} \lambda_\text{ext} ^T \diag^*(\gamma\mathbf{1})\bigg) \cr
	&=& \Tr\bigg( \lambda_\text{ext} \lambda_\text{ext} ^T \big(\gamma + \frac{1}{N-1}  \diag^*(\gamma\mathbf{1})\big)\bigg)
	\end{eqnarray}
	Plugging in $\gamma = \frac{N}{N-1} \Lambda - \frac{1}{N-1} \diag^*(\Lambda \mathbf{1})\in D$, \eqref{extreme pt max} is therefore
	\begin{eqnarray}
	\label{ext pt max1}
	\min_{\gamma,\Lambda} &\ &\frac{N}{N-1}\Tr(\lambda_\text{ext} \lambda_\text{ext} ^T\Lambda ) \\
	\text{s.t.} &\ & \gamma = \frac{N}{N-1} \Lambda - \frac{1}{N-1} \diag^*(\Lambda \mathbf{1})\cr
	&\ &\Lambda \succeq 0,\ \Lambda \geq 0, \ \mathbf{1}^T \Lambda \mathbf{1} = 1,\ \diag(\Lambda) =\frac{1}{N} \Lambda \mathbf{1}.\nonumber
	\end{eqnarray}
	To show $\gamma_\text{ext}$ in \eqref{gamma ext} is the unique minimizer of \eqref{ext pt max1}, it suffices to show $\gamma_\text{ext}$ is the unique minimizer for 
	\begin{eqnarray}
	\label{ext pt max2}
	\min_{\gamma,\Lambda} &\ &\frac{N}{N-1}\Tr(\lambda_\text{ext} \lambda_\text{ext} ^T\Lambda ) \\
	\text{s.t.} &\ & \gamma = \frac{N}{N-1} \Lambda - \frac{1}{N-1} \diag^*(\Lambda \mathbf{1}),\ \Lambda \succeq 0,\ \Tr(\Lambda) = 1/N,\nonumber
	\end{eqnarray}
	since the domain of \eqref{ext pt max1} is contained within \eqref{ext pt max2}. It is clear
    that the unique minimizer to \eqref{ext pt max2} is $\Lambda = \lambda_\text{ext}
    \lambda_\text{ext}^T$, implying that $\gamma_\text{ext}$ is the unique minimizer.
\end{proof}

\section{Tightening the convex relaxation}\label{section: 3mar}
Though Theorem \ref{theorem: extreme} shows that our convex relaxation with the $2$-marginals also contains $\tilde \Gamma_2$ as the extreme points, it may contain orther extreme points that do not come $\tilde \Gamma_2$. To further restrict the domain of optimization in SDP-Coulomb, one can consider applying convex relaxation to the $k$-marginals. In this section, we focus on the case of the 3-marginals. Let
\begin{equation}
\bar C(i,j,k) = C(i,j)+C(j,k)+C(k,i),\quad i,j,k=1,\ldots,\vert X \vert.
\end{equation}
Let the $N$-representable 3-marginal of $\mu$ be $\kappa$. In terms of $\bar C$ and $\kappa$, the cost of \eqref{symmetric MM} becomes
\begin{equation}
\label{3mar MM}
g(\lambda)+  \frac{N(N-1)(N-2)}{6} \sum_{i,j,k=1}^{\vert X \vert} \bar C(i,j,k)\kappa(i,j,k).
\end{equation}
In the following sections, we work out the domain of $\kappa$ in order to perform minimization. We
follow the derivation in \cite{friesecke2018breaking} in which the set $\Gamma_2$ is derived.


\subsection{The extreme points of the symmetric discrete distribution on $X^N$}
Let the set of symmetric discrete $N$-marginal be defined as
\begin{equation}
  \Pi_{N,\text{sym}} = \left\{ \mu\in ({\mathbb{R}^{\vert X \vert}})^N \ \vert\ \mu\ \text{is
    symmetric}, \mu\geq 0, \sum_{i_1,\ldots,i_N = 1}^{\vert X \vert} \mu(i_1,\ldots,i_N) = 1\right\}.
\end{equation}
Let $e_l \in \mathbb{R}^{\vert X \vert}$ be defined as $e_l(j)=\delta_{lj}$. For the set of
probability measures on $X^N$, an extreme point is
\begin{equation}
  \label{extreme Nd}
  e_{c_1}\otimes \ldots \otimes e_{c_N},
\end{equation}
for some $c_1,\ldots,c_N \in \{1,\ldots, \vert X \vert \}$. Therefore for the set of symmetric
measure $\Pi_{N,\text{sym}}$, an extreme point can be obtained from symmetrizing \eqref{extreme Nd},
giving rise to the set
\begin{equation}
  \Gamma_N = \left\{\frac{1}{N!} \sum_{\sigma\in S(N)} e_{c_{\sigma(1)}} \otimes \ldots \otimes
  e_{c_{\sigma(N)}} \ \vert \ c_1,\ldots, c_N \in \{1,\ldots,\vert X \vert \}\right\}
\end{equation}
where $S(N)$ is the symmetric group over $N$ numbers. For physical measure of the electrons, we look at a restricted set 
\begin{equation}
\tilde \Pi_{N,\text{sym}} =\left\{ \mu\in \Pi_{N,\text{sym}} \ \vert\ \mu(i_1,\ldots,i_N) = 0, \ \text{if}\ i_k = i_l\ \forall k,l = 1,\ldots,N\right \}
\end{equation}
which ensures two electrons cannot be in the same state. A derivation similar to Corollary \ref{corollary: extreme point} reveals that
\begin{equation}
\conv(\tilde \Gamma_N) = \tilde \Pi_{N,\text{sym}} 
\end{equation}
where
\begin{equation}
\tilde \Gamma_N = \left\{\frac{1}{N!} \sum_{\sigma\in S(N)} e_{c_{\sigma(1)}}  \otimes \ldots \otimes e_{c_{\sigma(N)}} \ \vert \ c_1,\ldots, c_N \in \{1,\ldots,\vert X \vert \},\ c_i\neq c_j \ \forall\ i,j\in N,\ i\neq j\right\}.
\end{equation}

\subsection{Convex hull of the set of $N$-representable 3-marginals}
To get a description to the set of $N$-representable 3-marginals in order to restrict $\kappa$ in
\eqref{3mar MM}, we marginalize the measures in $\tilde \Pi_{N,\text{sym}}$. Since $ \tilde
\Pi_{N,\text{sym}} = \conv(\tilde \Gamma_N)$, it suffices to marginalize the elements in $\tilde
\Gamma_N$. Picking an arbitrary element in $\tilde \Gamma_N$, then its 3-marginal is
\begin{eqnarray}
\label{3mar ext1}
&\ & \frac{1}{N!} \sum_{\sigma\in S(N) } \sum_{l_4,\ldots,l_N=1}^{\vert X \vert}e_{c_{\sigma(1)}}\otimes \ldots \otimes e_{c_{\sigma(N-1)}}(l_{N-1})  \otimes e_{c_{\sigma(N)}}(l_N) \cr
&=& \frac{1}{N!} \sum_{\sigma\in S(N) } e_{c_{\sigma(1)}} \otimes e_{c_{\sigma(2)}} \otimes e_{c_{\sigma(3)}} \cr
&=& \frac{(N-3)!}{N!} \sum_{\substack{i,j,k=1\\ i>j>k}}^N e_{c_i} \otimes e_{c_j}\otimes e_{c_k}\cr
&=& \frac{1}{N(N-1)(N-2)} \bigg(\sum_{i,j,k=1}^N  e_{c_i} \otimes e_{c_j}\otimes e_{c_k}+ 2 \sum_{k=1}^N e_{c_k} \otimes e_{c_k} \otimes e_{c_k} \cr 
&\ & -\sum_{i,j=1}^N  e_{c_i}\otimes e_{c_j} \otimes e_{c_j} -\sum_{i,j=1}^N  e_{c_j}\otimes e_{c_i} \otimes e_{c_j} -\sum_{i,j=1}^N  e_{c_i}\otimes e_{c_i} \otimes e_{c_j} \bigg).
\end{eqnarray}
The second equality follows from the fact that there are $(N-3)!$ $\sigma\in S(N)$ such that
$e_{c_i}\otimes e_{c_j} \otimes e_{c_k} = e_{c_{\sigma(1)}} \otimes e_{c_{\sigma(2)}} \otimes
e_{c_{\sigma(3)}}$ for a fixed $e_{c_i}\otimes e_{c_j} \otimes e_{c_k} $. Letting
\begin{equation}
  \label{lambda def2}
  \lambda := \frac{1}{N}\sum_{i=1}^N e_{c_i},
\end{equation}
it follows that $\lambda\in \{0,1/N\}^{\vert X \vert}$, and $\lambda^T \mathbf{1}=1$, since each
$e_{c_i}$ has only an entry with value 1 and is 0 everywhere else, and $c_i\neq c_j$ for all $i\neq
j$. Moreover,
\begin{equation}
\label{3mar ext2}
\frac{1}{N}\sum_{i=1}^{N} e_{c_i}(l)e_{c_i}(l)e_{c_i}(l) = \frac{1}{N}\sum_{i=1}^{N}e_{c_i}(l)e_{c_i}(l) = \frac{1}{N}\sum_{i=1}^{N}e_{c_i}(l)= \lambda(l),\quad l=1,\ldots,\vert X \vert,
\end{equation}
and 
\begin{equation}
\label{3mar ext3}
\frac{1}{N}\sum_{i=1}^{N} e_{c_i}(l)e_{c_i}(j) = 0,\quad \frac{1}{N}\sum_{i=1}^{N} e_{c_i}(l)e_{c_i}(j)e_{c_i}(k) = 0\quad \text{if}\ l\neq j, \text{or}\ j\neq k,\ \text{or}\ k\neq l .
\end{equation}
Writing \eqref{3mar ext1} in terms of $\lambda$ using \eqref{3mar ext2} and \eqref{3mar ext3}, one
can marginalize $\tilde \Gamma_N$ to obtain 
\begin{eqnarray}
  \tilde \Gamma_3 &=& \bigg \{ \frac{1}{N(N-1)(N-2)} \big( N^3 \lambda\otimes \lambda \otimes\lambda
  + 2N \sum_{l=1}^{\vert X \vert} \lambda(l) e_l \otimes e_l \otimes e_l - N^2 \sum_{l=1}^{\vert X
    \vert} \lambda(l) \lambda \otimes e_l \otimes e_l \cr
  &\ & - N^2 \sum_{l=1}^{\vert X \vert} \lambda(l) e_l \otimes \lambda \otimes e_l - N^2
  \sum_{l=1}^{\vert X \vert} \lambda(l) e_l \otimes e_l\otimes \lambda \big )\ \vert \ \lambda\in
  \mathcal{B}_N(X) \bigg \}.
\end{eqnarray}
Since every physical $N$-representable 3-marginal comes from the marginalization of an element in $\tilde
\Pi_{N,\text{sym}}=\conv(\tilde \Gamma_N)$, the following statement holds.
\begin{proposition}
  The set of $N$-representable 3-marginals coming from the marginalization of $\tilde \Pi_{N,\text{sym}}$ is $\conv(\tilde \Gamma_3)$.
\end{proposition}

With this proposition, in order to minimize \eqref{3mar MM} one can solve
\begin{eqnarray}
  \label{3mar MM2}
  \min_{\substack{\kappa \in \mathbb{R}^{\vert X \vert \times \vert X \vert  \times \vert X \vert },\\ \lambda\in \mathbb{R}^{\vert X \vert }} } 
  &\ & g(\lambda)+  \frac{N(N-1)(N-2)}{6} \sum_{i,j,k=1}^{\vert X \vert} \bar C(i,j,k)\kappa(i,j,k)\\
  \text{s.t.}&\ & \kappa\in \conv(\tilde \Gamma_3)\cr
  &\ & \lambda(i) = \sum_{j,k=1}^{\vert X \vert} \kappa(i,j,k),\quad i=1,\ldots,\vert X \vert\cr
  &\ & \mathcal{A}(\lambda) = b.\nonumber
\end{eqnarray}

\subsection{Convex relaxation to the 3-marginal problem}
The variable $\kappa$ in \eqref{3mar MM2} takes the form $\kappa = \sum_{i=1}^{m} a(i)
\kappa_i,\kappa_i \in \tilde \Gamma_3$ with $a \geq 0$ and $a^T \mathbf{1} =1$. Therefore, in order
to derive a convex relaxation to \eqref{3mar MM2}, one seeks a convex set that contains all the
elements in $\tilde \Gamma_3$. Such a set will certainly contain $\kappa = \sum_{i=1}^{m}
a(i)\kappa_i$, which is a convex combination of $\kappa_i \in \tilde \Gamma_3,i=1,\ldots,m$. For
this purpose, let
\begin{equation}
  \Theta := \lambda \otimes \lambda \otimes \lambda,\quad \lambda \in \mathcal{B}_N(X).
\end{equation}
Since $\lambda^T \mathbf{1} = 1$,
\begin{equation}
  \lambda \lambda^T = \sum_{k=1}^{\vert X \vert } \Theta(:,:,k),\quad \lambda  = \sum_{j,k=1}^{\vert X \vert } \Theta(:,j,k).
\end{equation}
Then in terms of $\Theta$, an extreme point $\kappa\in\tilde \Gamma_3$ is
\begin{multline}
\kappa = \phi(\Theta):=\frac{1}{N(N-1)(N-2)} \bigg(N^3\Theta + 2N \sum_{l=1}^{\vert X \vert }\big(\sum_{j,k=1}^{\vert X \vert } \Theta(l,j,k)\big) e_l \otimes e_l \otimes e_l\cr
  -N^2  \sum_{l=1}^{\vert X \vert } \big(\sum_{k=1}^{\vert X \vert } \Theta(l,:,k)\big)\otimes e_l \otimes e_l  -N^2  \sum_{l=1}^{\vert X \vert }e_l \otimes \big(\sum_{k=1}^{\vert X \vert } \Theta(l,:,k)\big)\otimes e_l \cr
 -N^2  \sum_{l=1}^{\vert X \vert }e_l \otimes e_l \otimes \big(\sum_{k=1}^{\vert X \vert } \Theta(l,:,k)\big) \bigg).
\end{multline}
Next, we impose some necessary conditions on $\Theta$ in a convex manner so that $\Theta$ comes from
the tensor product of the quantized marginals $\lambda$. Clearly, the symmetry property implies
\begin{equation}
\label{symm con}
\Theta(i,j,k) = \Theta(k,i,j) = \Theta(j,k,i) =  \Theta(j,i,k) = \Theta(k,j,i) = \Theta(i,k,j)
\end{equation}
Since $\lambda\in \{0,1/N\}^{\vert X \vert}$,
\begin{equation}
\label{quantized 2}
\lambda(i) \lambda(i) \lambda(j) = \lambda(i)\lambda(j)/N \Rightarrow \Theta(i,i,j) = \frac{1}{N}\sum_{k=1}^{\vert X \vert } \Theta(i,j,k), \quad \forall i,j = 1,\ldots,\vert X \vert.
\end{equation}
Then the constraint that $\lambda^T \mathbf{1} = 1$ gives
\begin{equation}
\label{normalize}
\sum_{i=1}^{\vert X \vert }  \lambda(i) = 1 \Rightarrow \sum_{i,j,k=1}^{\vert X \vert } \Theta(i,j,k) = 1.
\end{equation}
We also have the conic constraints
\begin{equation}
\label{slice psd}
\Theta(:,:,i) = \lambda\lambda^T \lambda(i) \succeq 0,\quad \forall i=1,\ldots,\vert X \vert.
\end{equation}
and 
\begin{equation}
\label{positivity}
\Theta\geq 0.
\end{equation}
Combining \eqref{symm con},\eqref{quantized 2},\eqref{normalize},\eqref{slice
  psd} and \eqref{positivity} leads to the following optimization problem over $\Theta$.
\begin{eqnarray}
  \label{SDPCoulomb2}
  \min_{\Theta,\kappa \in \mathbb{R}^{\vert X \vert \times \vert X \vert  \times \vert X \vert }}  &\ & g(\sum_{j,k=1}^{\vert X \vert} \Theta(:,j,k))+  \frac{N(N-1)(N-2)}{6} \sum_{i,j,k=1}^{\vert X \vert} \bar C(i,j,k)\kappa(i,j,k)\quad (\text{SDP-Coulomb2})\\
\text{s.t.}&\ & \kappa = \phi(\Theta)\cr
&\ &\Theta\ \text{is symmetric}\cr
&\ &  \Theta(i,i,j) = \frac{1}{N}\sum_{k=1}^{\vert X \vert } \Theta(i,j,k), \quad \forall i,j = 1,\ldots,\vert X \vert\cr
&\ & \sum_{i,j,k=1}^{\vert X \vert } \Theta(i,j,k) = 1\cr
&\ & \Theta(:,:,i) \succeq 0\ \forall i=1,\ldots,\vert X\vert,  \Theta \geq 0 \cr
&\ & \mathcal{A}(\sum_{j,k=1}^{\vert X \vert} \Theta(:,j,k)) = b.\nonumber
\end{eqnarray}

\subsection{A remark on Lassere's hierarchy}
It is possible to use the Lassere hierarchy (or sum-of-squares hierarchy) \cite{anjos2012introduction, blekherman2012semidefinite} to further tighten the convex relaxation. When applying this method to our problem, the task of determining some power of the quantized 1-marginal $\lambda\in \mathcal{B}_N(X)$ (for example the problem of determining the 2 and 3-marginals), is reformulated as a moment determination problem. More precisely, instead of working with the monomials $\{\lambda^\alpha\}_\alpha$ where $\alpha \in \mathbb{N}^{\vert X \vert}$ is a multi-index and $\mathbb{N}$ is the set of natural numbers, one performs a change of variables according to
\begin{equation}
[\lambda^\alpha \lambda^\beta]_{\alpha,\beta} \Rightarrow [\mathbb{E}(\lambda^\alpha \lambda^\beta)]_{\alpha,\beta}.
\end{equation} 
The optimization variable, the matrix $[\mathbb{E}(\lambda^\alpha \lambda^\beta)]_{\alpha,\beta}$, has size ${\ p+\vert X \vert \choose p}$ for each dimension if we consider the monomials $\lambda^\alpha$'s and $\lambda^\beta$'s up to degree $p$.  Then, an equality constraint $h(\lambda)=0$ ($h$ is a polynomial) is changed according to
\begin{equation}
h(\lambda) = 0 \Rightarrow \mathbb{E}(h(\lambda) \lambda^\alpha) = 0\ \forall \alpha,
\end{equation} 
and an inequality constraint $q(\lambda)=0$ ($q$ is a polynomial) is changed according to
\begin{equation}
q(\lambda) \geq 0 \Rightarrow \mathbb{E}(q(\lambda) s(\lambda)^2) \geq 0 \ \forall s(\lambda)
\end{equation} 
where $s$ is some polynomial. The inequality constraints leads to a positive semidefinite constraint. For example the constraint $\lambda \geq 0$ simply gives
\begin{equation}
v^T \left( [\mathbb{E}(\lambda^\alpha \lambda^\beta)]_{\alpha,\beta} \right) v \geq 0, \quad \forall v \ \text{with size}\  {\ p+\vert X \vert \choose p},
\end{equation}
if we consider the monomials $\lambda^\alpha$'s and $\lambda^\beta$'s up to degree $p$. As can be
seen, when choosing $p\geq2$, we already face with $\vert X \vert^4$ variables. Therefore, we pursue
a cheaper alternative.
\section{Rounding}
\label{section: implementation}
The previous sections describe several convex relaxation approaches for solving the multi-marginal
transport problem. The general philosophy is to enlarge the domain of optimization, therefore
obtaining a lower bound for the global minimum. To obtain an upper bound for the global minimum, we
need to project the solution back into the unrelaxed domain ($\conv(\tilde \Gamma_2$) or
$\conv(\tilde \Gamma_3$)). We consider two cases of practical importance:
\begin{enumerate}
\item When the linear constraint $\mathcal{A}(\lambda) = b$ is not present in \eqref{symmetric MM}.
\item When $\mathcal{A}(\lambda) = b \rightarrow \lambda = \rho$, for example
  when solving the multimarginal-optimal transport problem \eqref{SCE OT}.
\end{enumerate}
Section \ref{section: rounding 1} addresses the first case. Here, we devise a scheme to round the
solution from SDP-Coulomb to the set of extreme points $\tilde \Gamma_2$ for the set of
$N$-representable 2-marginals. In Section \ref{section: rounding 2}, we deal with the second case
with the marginal constraint. For this case, it is difficult to work with SDP-Coulomb to obtain a
rounded solution in $\tilde \Gamma_2$. Therefore, we discuss how we can use SDP-Coulomb2 for such a
purpose.

\subsection{Without the linear constraint $\mathcal{A}(\lambda) = b$}
\label{section: rounding 1}
In the special case where the constraint $\mathcal{A}(\Lambda \mathbf{1}) = b$ is absent and
$g(\cdot)$ is a linear functional, we simply minimize a linear functional of $\Lambda$ in
SDP-Coulomb. In principle, if the domain of SDP-Coulomb (without $\mathcal{A}(\Lambda \mathbf{1}) =
b$) is close to the set of $N$-representable 2-marginals with zero diagonal ($\conv(\tilde
\Gamma_2)$) in Corollary \ref{corollary: extreme point}), then SDP-Coulomb should return a solution
$\Lambda^* \approx \lambda^* {\lambda^*}^T $ where $\lambda^*\in \mathcal{B}_N(X)$. This is because
the extreme points of $\conv(\tilde \Gamma_2)$ is $\tilde \Gamma_2$ (Corollary \ref{corollary:
  extreme point}), and generically, the optimizer of a linear functional over a convex set is an
extreme point of the set. We therefore propose a rounding procedure in Alg. \ref{algorithm:
  rounding}. If SDP-Coulomb returns a solution $\Lambda^*$ where the entries on the diagonal of
$\Lambda^*$ are not exactly $1/N^2$ or $0$, letting the index of the largest entry of
$\diag(\Lambda^*)$ be $i_\text{max}$, we add a linear constraint $\diag(\Lambda)(i_\text{max}) =
1/N^2$ to SDP-Coulomb. This step is repeated until a rank-1 $\Lambda^*$ is obtained. This is
summarized in Alg. \ref{algorithm: rounding}.
\begin{algorithm}
  \caption{Rounding in the absence of the linear constraint $\mathcal{A}(\lambda)=b$}\label{algorithm: rounding}
  \begin{algorithmic}[1]
	\Procedure{Rounding}{}
	\State $\Lambda^*\gets$ Solution to SDP-Coulomb.
	\State  $\mathcal{I} \gets \{\emptyset\}$, $R\gets I$ 
	\While {$\text{rank}(\Lambda^*)>1$}
	\State $i_\text{max} \leftarrow$ index of the largest element in $R\diag(\Lambda^*)$.
	\State $\mathcal{I} \gets \mathcal{I}\cup i_\text{max}$, $\mathcal{I}^c \gets \{1,\ldots,\vert X \vert\}\setminus \mathcal{I}$.
	\State $R\gets$ $I(\mathcal{I}^c,:)$.
	\State $\Lambda^* \gets$ Solution to SDP-Coulomb with the extra constraint $\diag(\Lambda)_\mathcal{I} = 1/N^2$.
	\EndWhile
	\State \Return  $\Lambda^*$.
	\EndProcedure
  \end{algorithmic}
\end{algorithm}
We remark that this procedure is crucial when there are degenerate solutions, giving a high rank
solution in SDP-Coulomb.

\subsection{With the marginal constraint $\lambda = \rho$}
\label{section: rounding 2}
When having the constraint $\Lambda \mathbf{1} = \rho$ in SDP-Coulomb, we cannot pursue the same
strategy as in Section \ref{section: rounding 1} to round the solution. When there exists a marginal
constraint, we expect the solution to \eqref{2mar MM} to be a convex combination of the extreme
points from $\tilde \Gamma_2$, implying SDP-Coulomb returns solution as $\Lambda^* \approx
\sum_{i=1}^{m} a^*(i)\lambda^*_i {\lambda^*_i}^T, {a^*}^T \mathbf{1} = 1, a^*\geq 0$. However, in
order to round, one has to first disentangle each $\lambda_i^*$ from such a convex
combination. Since $\lambda_i^*$'s are not orthogonal to each other,
it is not obvious how one can use matrix factorization techniques such as an eigendecomposition to
obtain the $\lambda_i^*$'s from $\Lambda^*$. To this end, we resort to using SDP-Coulomb2 to obtain
each $\lambda_i^*$. Since in SDP-Coulomb2, we expect to have the solution $\Theta^* \approx
\sum_{i=1}^{m} a^*(i) \lambda_i^* \otimes \lambda_i^* \otimes \lambda_i^*, \lambda_i^*\in
\mathcal{B}_N(X)$ (as we expect the solution to approximately lie in $\conv(\tilde \Gamma_3)$), we
resort to using a CP-tensor decomposition \cite{kolda2009tensor} to obtain each individual $\lambda_i^*$
approximately.

In order to use a CP-decomposition, one needs to have an idea of what $m$ is. The following
discussion demonstrates that $m=|X|$. We first look at the set of the physical symmetric probability measures on $X^N$ that have the marginal being $\rho$:
\begin{eqnarray}
\tilde \Pi_{N,\text{sym}}(\rho) &=& \{ \mu\in \tilde \Pi_{N,\text{sym}} \ \vert\  \sum_{i_2,\ldots,i_N=1}^{\vert X \vert}\mu(:,i_2\ldots,i_N) = \rho\}\cr
 &=& \conv(\tilde \Gamma_N)\cap \bigg\{ \mu\in ({\mathbb{R}^{\vert X \vert}})^N \ \vert  \sum_{i_2,\ldots,i_N=1}^{\vert X \vert} \mu(i_1,i_2\ldots,i_N) = \rho(i_1),\ i_1 = 1,\ldots,\vert X \vert -1 \bigg\}.\label{Nmar ext}
\end{eqnarray}
Notice that the marginal constraint in \eqref{Nmar ext} is only enforced for $\vert X \vert -1$
sites. This is because for $\mu \in \conv(\tilde \Gamma_N)$,
\begin{equation}
\sum_{i_2,\ldots,i_N=1}^{\vert X \vert} \mu(\vert X \vert,i_2\ldots,i_N)
\end{equation}
is completely determined by 
\begin{equation}
\sum_{i_2,\ldots,i_N=1}^{\vert X \vert} \mu(i_1,i_2\ldots,i_N),\quad i_1 = 1,\ldots,\vert X \vert-1
\end{equation}
via
\begin{equation}
\sum_{i_2,\ldots,i_N=1}^{\vert X \vert} \mu(\vert X \vert,i_2\ldots,i_N) = 1 - \sum_{i_1=1}^{\vert X
  \vert -1} \sum_{i_2,\ldots,i_N=1}^{\vert X \vert} \mu(i_1,i_2\ldots,i_N).
\end{equation}
We now appeal to the results in \cite{dubins1962extreme} to see what $m$ is. The theorem in
\cite{dubins1962extreme} implies that for a closed and bounded convex set $\mathcal{K}$, an extreme point of
$\mathcal{K}\cap H_1 \cap \cdots H_n$ where $H_1,\ldots,H_n$ are $n$ hyperplanes can be represented as $n+1$
convex combination of the extreme points of $\mathcal{K}$. Since $\tilde \Pi_{N,\text{sym}}(\rho)$ in
\eqref{Nmar ext} is the intersection of $\conv(\tilde \Gamma_N)$ with $\vert X \vert -1$
hyperplanes, it follows that for an extreme point $\mu \in \tilde \Pi_{N,\text{sym}}(\rho)$, $\mu$ is the convex combination of $\vert X \vert$ elements in $\tilde \Gamma_N$. After a marginalization, it follows that a physical $N$-representable 3-marginal that satisfies the marginal constraint is a convex combination of $\vert X \vert$ elements of $\tilde \Gamma_3$, therefore $m=\vert X \vert$.

As $\Theta^* \approx \sum_{i=1}^{\vert X\vert} a^*(i) \lambda_i^* \otimes \lambda_i^* \otimes
\lambda_i^*$, if the approximation $\approx$ holds with an $=$ sign, and if $\lambda_1^*,\ldots,\lambda^*_{\vert X \vert}$ are linearly independent, then $\Theta^*$ has a unique CP tensor decomposition, up to ordering and magnitude of
$\lambda^*_i$'s. This can be seen in Section \ref{section: jenrich} where Jenrich's algorithm
provides an explicit construction of the $\lambda_i^*$'s. We note that although the assumption of linearly
independent $\lambda_1^*,\ldots,\lambda^*_{\vert X \vert}$ is required for the success of Jenrich's
algorithm, it is not a necessary condition to ensure the uniqueness of the CP-decomposition (see for
example the theorem of Kruskal \cite{kruskal1977three}). In the situation where the linearly
independence assumption is violated, one may use a different algorithm such as the alternating
least-squares (ALS) for recovering the tensor components. Therefore, our rounding algorithm has
three phases. We first use Jenrich's algorithm to obtain an initialization for
$\lambda_i^*,i=1,\ldots,\vert X \vert$. Then a procedure based on ALS is used to refine the solution
from Jenrich's algorithm and also enlarge the set $\{\lambda_i^*\}_{i=1}^{\vert X \vert }$ to
$\{\lambda_i^*\}_{i=1}^p$. Lastly, we solve a regression problem to determine the convex combination
of $\{\lambda_i^*\}_{i=1}^p$ that approximate $\Theta^*$ while satisfying the marginal
constraint. The algorithm is summarized in Alg. \ref{algorithm: round2}.
\begin{algorithm}
  \caption{Algorithm for rounding  in the presence of the marginal constraint}\label{algorithm: round2}
  \begin{algorithmic}[1]
	\Procedure{Rounding2}{$\delta,\rho$}
	\State $\Theta^* \gets$ Solution to SDP-Coulomb2. 
	\State $\{\lambda_i^*\}_{i=1}^{\vert X \vert} \gets$ JENRICH($\Theta^*$) (Section \ref{section: jenrich}).
	\State $\{\lambda_i^*\}_{i=1}^p \gets$ ALS($\Theta^*,\{\lambda_i^*\}_{i=1}^{\vert X \vert},\delta$) (Section \ref{section: als}).
	\State $a^* \gets \underset{a\in\mathbb{R}^p}{\text{argmin}}  \|\Theta^*- \sum_{i=1}^p a(i) \lambda_i^* \otimes \lambda_i^* \otimes \lambda_i^* \|_F^2$ \ \ s.t. $a\geq 0$, $a^T \mathbf{1} =1 , \sum_{i=1}^p a(i) \lambda_i^* =\rho$.
	\State $\Theta^* \gets \sum_{i=1}^p a^*(i) \lambda_i^* \otimes \lambda_i^* \otimes \lambda_i^*$, 
	\State \Return $\Theta^*$
	\EndProcedure
  \end{algorithmic}
\end{algorithm}

\subsubsection{Jenrich's algorithm} \label{section: jenrich}
In this section, we provide the details for Jenrich's algorithm in Alg. \ref{algorithm: jenrich} for
the sake of completeness.
\begin{algorithm}
	\caption{Jenrich algorithm}\label{algorithm: jenrich}
	\begin{algorithmic}[1]
		\Procedure{Jenrich}{$\Theta$} 
		\State  Get $w_1,w_2 \in \mathbb{R}^{\vert X \vert}$,  $w_1(i),w_2(i) \sim \text{uniform}[0,1],i = 1,\ldots \vert X \vert$.
		\State $W_1 \gets \sum_{k=1}^{\vert X \vert} w_1(k) \Theta(:,:,k)$, $W_2 \gets \sum_{k=1}^{\vert X \vert} w_2(k) \Theta(:,:,k)$.
		\State Eigendecompose $W_1 W_2^{\dagger} = U \Sigma U^\dagger$, where $\Sigma$ is a diagonal matrix.
		\State $\lambda_i \gets U(:,i), i=1,\ldots,\vert X \vert$.
		\State $\lambda_i \gets \frac{\lambda_i}{\sqrt{N}\| \lambda_i\|_2}, i=1,\ldots,\vert X \vert$.
		\State \Return $\{\lambda\}_{i=1}^{\vert X \vert}$.
		\EndProcedure
	\end{algorithmic}
\end{algorithm}
The key idea of Alg. \ref{algorithm: jenrich} is that, if $\Theta = \sum_{i=1}^{\vert X \vert}
a(i) \lambda_i \otimes \lambda_i \otimes \lambda_i$, then
\begin{equation}
W_1 = \sum_{i=1}^{\vert X \vert} (a(i) w_1^T \lambda_i) \lambda_i \lambda_i^T,\quad W_2 = \sum_{i=1}^{\vert X \vert} (a(i) w_2^T \lambda_i) \lambda_i \lambda_i^T.
\end{equation}
Thus 
\begin{equation}
\label{Jenrich eig}
W_1 W_2^\dagger = U S U^\dagger,\quad U =[ \lambda_1 \cdots  \lambda_{\vert X \vert} ],\quad \Sigma = \diag^*\bigg(\bigg[\frac{a(1) w_1^T \lambda_1}{a(1) w_2^T \lambda_1},\ldots,\frac{a(\vert X \vert) w_1^T \lambda_{\vert X \vert}}{a(\vert X \vert) w_2^T \lambda_{\vert X \vert}}\bigg]\bigg).
\end{equation}
So the eigenvectors of $W_1 W_2^\dagger$ give $\lambda_1\ldots,\lambda_{\vert X \vert}$. The last
step in Alg. \ref{algorithm: jenrich} is a normalization step to ensure $\|\lambda_i\|=1/\sqrt{N}$
for all $i$, since in principle $\lambda_i\in \mathcal{B}_N(X)$. As we see, if in \eqref{Jenrich
  eig} $\lambda_1,\ldots,\lambda_{\vert X \vert}$ are linearly independent, Jenrich's algorithm
gives a unique decomposition since $\text{diag}(\Sigma)$ is non-degenerate generically (except for
the entries correspond to $a(i)=0$).

\subsubsection{Alternating least-squares} \label{section: als}
To further refine the solution from Jenrich's algorithm to approximate a given tensor $\Theta$, we
propose to use a variant of the ALS that is similar to a projected gradient descent.  Ideally, if
$\Theta = \sum_{i=1}^{\vert X \vert} a(i) \lambda_i\otimes \lambda_i \otimes \lambda_i$, one can try
to solve
\begin{eqnarray}
\label{ALS}
\min_{\substack{a\in \mathbb{R}^{\vert X \vert},\\ P,Q,R\in \mathbb{R}^{\vert X \vert\times \vert X \vert}}}&\ & \| \sum_{i=1}^{\vert X \vert} P(:,i) \otimes Q(:,i) \otimes R(:,i) - \Theta \|_F^2\\
\text{s.t.}&\ & Q=R, P =R\diag^*(a)\cr
&\ &a\geq 0,\ a^T\mathbf{1} = 1\cr
&\ & R(:,i) \in \mathcal{B}_N(X).\nonumber
\end{eqnarray}
using a local optimization algorithm and
identify the $\lambda_i$'s with the $R(:,i)$'s, provided Jenrich's algorithm gives a good
initialization. There is however a caveat. Although $\sum_{i=1}^{\vert X \vert} P(:,i) \otimes
Q(:,i) \otimes R(:,i)$ provides an approximation to the 3-marginal $\Theta$, $\sum_{k,j=1}^{\vert X
  \vert}\sum_{i=1}^{\vert X \vert} P(:,i) \otimes Q(k,i) \otimes R(j,i) \neq \rho$ in general, hence
the marginal constraint can be violated. To deal with such an issue, we want to identify a set of
$\lambda_i$'s in $\mathcal{B}_N(X)$, $\{\lambda_i\}_{i=1}^p$, where $p>\vert X \vert$.  With a more
generous selection of the $\lambda_i$'s, some convex combination of $\{\lambda_i\}_{i=1}^p$ should give
the correct marginal while approximating $\Theta$ from SDP-Coulomb2 \eqref{SDPCoulomb2}.

To this end, the following problem with a less stringent constraint is solved instead:
\begin{eqnarray}
\label{nonsymmetric ALS}
\min_{P,Q,R\in \mathbb{R}^{\vert X \vert\times \vert X \vert}}&\ & \| \sum_{i=1}^{\vert X \vert} P(:,i) \otimes Q(:,i) \otimes R(:,i) - \Theta \|_F^2\\
\text{s.t.}&\ & \|Q(:,i)\|_2 = 1/\sqrt{N}\cr
&\ &N\ \text{entries of}\ \vert R(:,i) \vert \ \text{are}\ 1/N\ i=1,\ldots,\vert X \vert\nonumber.
\end{eqnarray}
Notice that each of the $R(:,i)$'s is not required to have only $N$ nonzero entries, unlike in
\eqref{ALS} where $R(:,i)$'s belong to $\mathcal{B}_N(X)$. To solve \eqref{nonsymmetric ALS}, we
use an ALS procedure detailed in Alg. \ref{algorithm: als}. The outer-loop of this procedure
controls the number of the entries of $R(:,i)$ that have magnitude $1/N$. At every step of
Alg. \ref{algorithm: als}, each column of $Q$ is normalized to $1/ \sqrt{N}$ after solving the
least-squares concerning $Q$. To enforce the constraint on $R(:,i)$ in \eqref{nonsymmetric ALS},
after solving the least-squares concerning $R$, for each $R(:,i)$, $k$ entries with the largest
magnitude are picked out and have their magnitude being set to $1/N$. When the iteration converges,
we then enforce $k+1$ entries of each $R(:,i), i=1,\ldots,\vert X \vert$ to have magnitude $1/N$ in
the ALS. These steps are repeated until $k=N$. We expect each $R(:,i), i=1,\ldots,\vert X \vert$ to
have $N$ or slightly greater than $N$ entries that are large in magnitude. Using the large magnitude
entries in each column of $R$, we exhaustively enumerate the candidate $\{\lambda_i\}_{i=1}^p$ where
$\lambda_i \in \mathcal{B}_N(X)$. The number $p$ is controlled via the parameter $\delta$.
\begin{algorithm}
  \caption{Modified alternating least-squares}\label{algorithm: als}
  \begin{algorithmic}[1]
	\Procedure{ALS}{$\Theta,\{\lambda_i\}_{i=1}^{\vert X \vert}$,$\delta$} 
	\State  Initialize $Q= [\lambda_1,\ldots,\lambda_{\vert X \vert}]$, $R = [\lambda_1,\ldots,\lambda_{\vert X \vert}]$.		
	\For{$k$ from 1 to $N$}
	\While{not converge}
	\State $P \gets \arg\min_{\tilde P\in \mathbb{R}^{\vert X \vert\times \vert X \vert}} \| \sum_{i=1}^{\vert X \vert} \tilde P(:,i) \otimes Q(:,i)  \otimes R(:,i)  - \Theta \|_F^2 $.
	\State $Q \gets \arg\min_{\tilde Q\in \mathbb{R}^{\vert X \vert\times \vert X \vert}} \| \sum_{i=1}^{\vert X \vert} P(:,i) \otimes \tilde Q(:,i)  \otimes R(:,i)  - \Theta \|_F^2 $.
	\State $Q(:,i) \gets \frac{Q(:,i)}{\sqrt{N} \|Q(:,i)\|_2},\ i=1,\ldots,N$.
	\State $R \gets \arg\min_{\tilde R\in \mathbb{R}^{\vert X \vert\times \vert X \vert}} \| \sum_{i=1}^{\vert X \vert} P(:,i) \otimes  Q(:,i)  \otimes\tilde R(:,i) - \Theta \|_F^2 $.
	\State Set $k$ entries of $R(:,i), i=1,\ldots,\vert X \vert$ with the largest magnitude to have magnitude $1/N$.
	\EndWhile
	\EndFor
	\For{$i$ from 1 to $\vert X \vert$}
	\State $\mathcal{I}_i \gets \{j\ \vert \  \vert C(j,i)\vert > \delta/N   \}$.
	\State Form $\xi_l^{(i)} \in \mathcal{B}_N(X), l=1,\ldots, {\vert \mathcal{I}_i \vert \choose N}$. The non-zero entries of $\xi_l^{(i)}$ for each $l$ are indexed by 
	\State a subset of $\mathcal{I}_i $ with $N$ elements.
	\State $p_i\gets {\vert \mathcal{I}_i \vert \choose N}$.
	\EndFor 
	\State $\{\lambda_i\}_{i=1}^p \gets \cup_{i=1}^{\vert X \vert} \{\xi_l^{(i)}\}_{l=1}^{p_i} $
	\State \Return $\{\lambda_i\}_{i=1}^p$.
	\EndProcedure
  \end{algorithmic}
\end{algorithm}

\section{Numerical simulations}\label{section: numerical}
In this section, we demonstrate the effectiveness of our approach using a few numerical examples. The energy is computed using
\begin{equation}
  E(\gamma) = \sum_{i,j=1}^{\vert X \vert} \Tr( C(i,j) \gamma(i,j)),
\end{equation}
where $\gamma$ is the 2-marginal, obtained either via SDP-Coulomb or SDP-Coulomb2 (or their rounded
versions). We denote the solution to SDP-Coulomb and SDP-Coulomb2 $\gamma_1^{-},\gamma_2^{-}$, and
their rounded solutions $\gamma_1^{+},\gamma_2^{+}$. The superscripts are used to indicate whether
we are using the solutions for the purpose of obtaining a lower bound or an upper bound for the
energy. We always choose $C$ such that $C(x,y) = \frac{1}{\| x - y \|_2}, x,y\in X$,. In all cases,
we choose a box $[-2,2]^d$ where $d$ is the dimension of the space where the electrons reside. A
uniform discretization is then applied to $[-2,2]^d$ to get the discrete domain $X$. We use
\begin{equation}
{E_\text{gap}}_i = \frac{E(\gamma_i^+) - E(\gamma_i^-) }{E(\gamma_i^-)},\quad i=1,2
\end{equation}
to provide an idea on how close we are to the true energy. SDP-Coulomb and SDP-Coulomb2 are
implemented using the large scale SDP solver SDPNAL+\cite{yang2015sdpnal}.

\subsection{Optimizing a linear functional over the 2-marginal}
In this section, we let $g(\lambda)$ in \eqref{symmetric MM} be an arbitrary linear functional $c^T
\lambda$. This can be seen as an external potential $v_\text{ext}$ in \eqref{total
  minimization}. Then SDP-Coulomb is solved to obtain the 2-marginals. Since one can already devise
a rounding scheme (Section \ref{section: rounding 1}) based on the solution of SDP-Coulomb, we only
present the energy gap derived from $\gamma_1^{-}$ and $\gamma_1^{+}$. Unlike SDP-Coulomb2,
SDP-Coulomb only involves a matrix with size $\vert X \vert \times \vert X \vert$, therefore we can
apply it to grids with larger size. The model for the vector $c$ considered is
\begin{equation}
c =\sigma (\min_{i,j}C(i,j)) \mathcal{N}(0,1).
\end{equation}
In Table \ref{table: egap2D} and \ref{table: egap3D}, we present ${E_\text{gap}}_1$ for $d=2,3$, with $N=5,9,13$. When $d=2$, we use a grid with size $\vert X \vert =20^2$. When $d=3$, we let $\vert X \vert = 9^3$. 

\begin{table}[ht]
	\centering 
	\begin{tabular}{c c c c } 
		\hline\hline 
		& $\sigma=0$ &$\sigma=0.25$  &  $\sigma=0.5$\\ [0.5ex] %
		\hline\hline 
		$n=5$  & 3.3e-03 & 7.6e-03 & 1.3e-02 \\
		$n=9$ & 3.8e-03 & 3.0e-03  &  3.6e-03\\ 
		$n=13$ & -2.0e-05 & 3.1e-03 & 3.4e-03  \\[1ex] 
		\hline 
	\end{tabular}
	\caption{${E_\text{gap}}_1$ for electrons in 2D. Here $\vert X \vert = 20^d$, $d=2$. The energy
      gap is averaged over 12 realizations of $c$. The negative gap between the upper and lower
      bounds when $\sigma=0, N=13$ is due to the accuracy limitation of the optimization package.
    }\label{table: egap2D}
\end{table}

\begin{table}[ht]
	\centering 
	\begin{tabular}{c c c c } 
		\hline\hline 
		& $\sigma=0$ &$\sigma=0.25$  &  $\sigma=0.5$\\ [0.5ex] %
		\hline\hline 
		$n=5$  & 3.7e-02&8.1e-03 & 5e-03   \\
		$n=9$ & 7.9e-03&5.1e-03 & 3.5e-03  \\ 
		$n=13$ & 3.2e-03&2.8e-03 & 3.1e-03  \\[1ex] 
		\hline 
	\end{tabular}
	\caption{${E_\text{gap}}_1$ for electrons in 3D. Here $\vert X \vert = 9^d$, $d=3$. The energy
      gap is averaged over 12 realizations of $c$. }\label{table: egap3D}
\end{table}

\subsection{Multimarginal Optimal Transport}
In this section, we present numerical results for different instances of Problem \eqref{SCE OT}. Both SDP-Coulomb and SDP-Coulomb2 are tested. Due to the size of the
variable in SDP-Coulomb2, we can only afford a smaller grid size. The point of the simulation is to
demonstrate how an upper bound of the energy can be extracted using SDP-Coulomb2, through method
presented in Section \ref{section: rounding 2}.

In the case of 1D, we use three different marginals
\begin{equation}
\rho_1(x) \propto 1,\ \quad  \rho_2(x) \propto \exp(-x^2/\sqrt{\pi}),\quad \rho_3(x) \propto \sin(4x)+1.5.
\end{equation}
where $\rho_1,\rho_2,\rho_3$ are appropriately normalized. Using the combination of SDP-Coulomb2 and Alg. \ref{algorithm: round2}, an upper-bound can
be obtained. We present the results with $\vert X \vert = 64$ and $N=8$ in Fig. \ref{figure: unif64},
\ref{figure: gaussian64} and \ref{figure: sine64}. In all examples, we obtained an energy gap from
the order of 1e-04 to 1e-02. The running times for SDP-Coulomb and SDP-Coulomb2 are about 7s and
249s on average. In general, we observe a fuzzier 2-marginal in SDP-Coulomb, especially when the
marginal is $\rho_3$. We note that the marginals chosen are bounded away from 0. This is because
if there are sites where the marginal is close to zero, due to the approximation error of
SDP-Coulomb2, $\Theta^*$ may be inaccurate on these sites, making rounding difficult. For the 2D case, we tested it on a Gaussian distribution
\begin{equation}
\rho_4(x,y) \propto \exp(-(x^2+y^2)/\sqrt{12\pi})
\end{equation}
with $\vert X\vert =10^2$ and $N=6$. The running time for SDP-Coulomb and SDP-Coulomb2 are 4.7s
and 731s respectively. Again, the difference between the quality of the solutions from SDP-Coulomb
and SDP-Coulomb2 is rather small.

\begin{figure}[!ht]
	\centering
	\subfloat[SDP-Coulomb.]{\includegraphics[width=0.47\columnwidth]{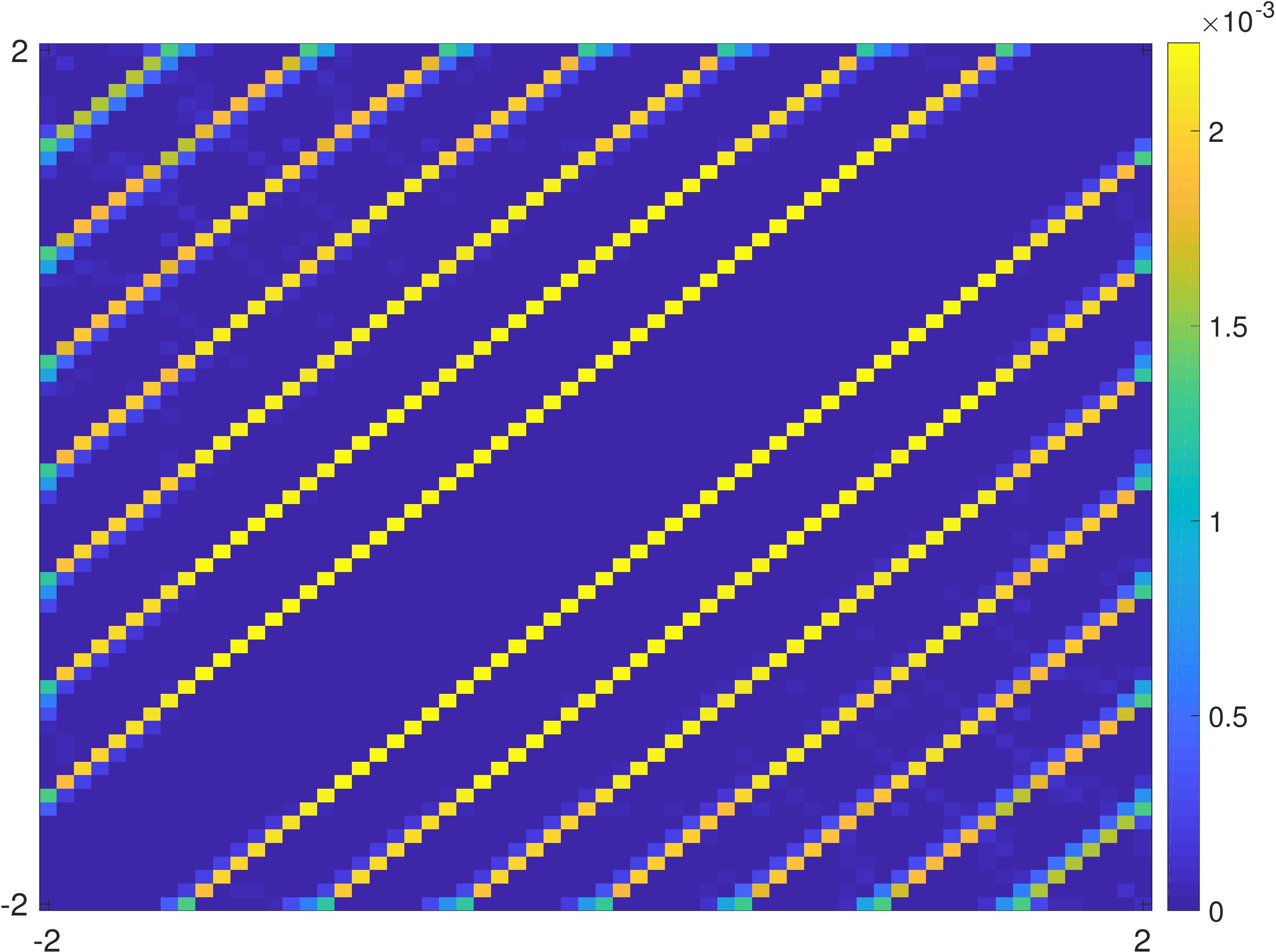}}
	\qquad
	\subfloat[SDP-Coulomb2.]{\includegraphics[width=0.47\columnwidth]{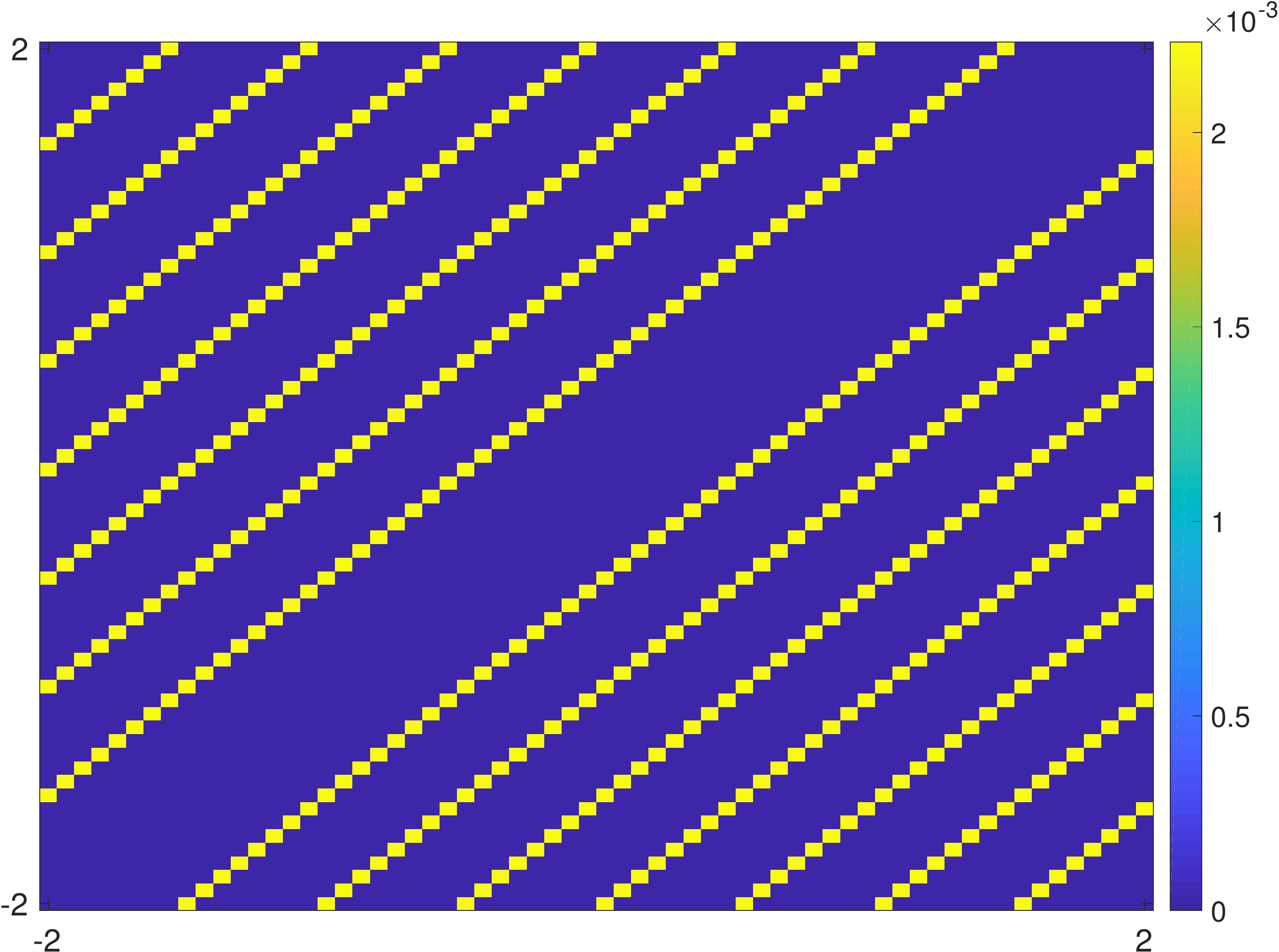}}
	\qquad
	\caption{2-marginal from solving the multimarginal transport problem with the marginal
      $\rho_1(x)$ where $N=8$, $\vert X \vert = 64$, $d=1$. (a): Solution from
      SDP-Coulomb. ${E_\text{gap}}_1=$ 4.9e-04. (b): Solution from SDP-Coulomb2. ${E_\text{gap}}_2=$
      -1.0e-06. The negative sign for the energy gap is due to the limitation of numerical
      accuracy. }\label{figure: unif64}
\end{figure}

\begin{figure}[!ht]
	\centering
	\subfloat[SDP-Coulomb.]{\includegraphics[width=0.47\columnwidth]{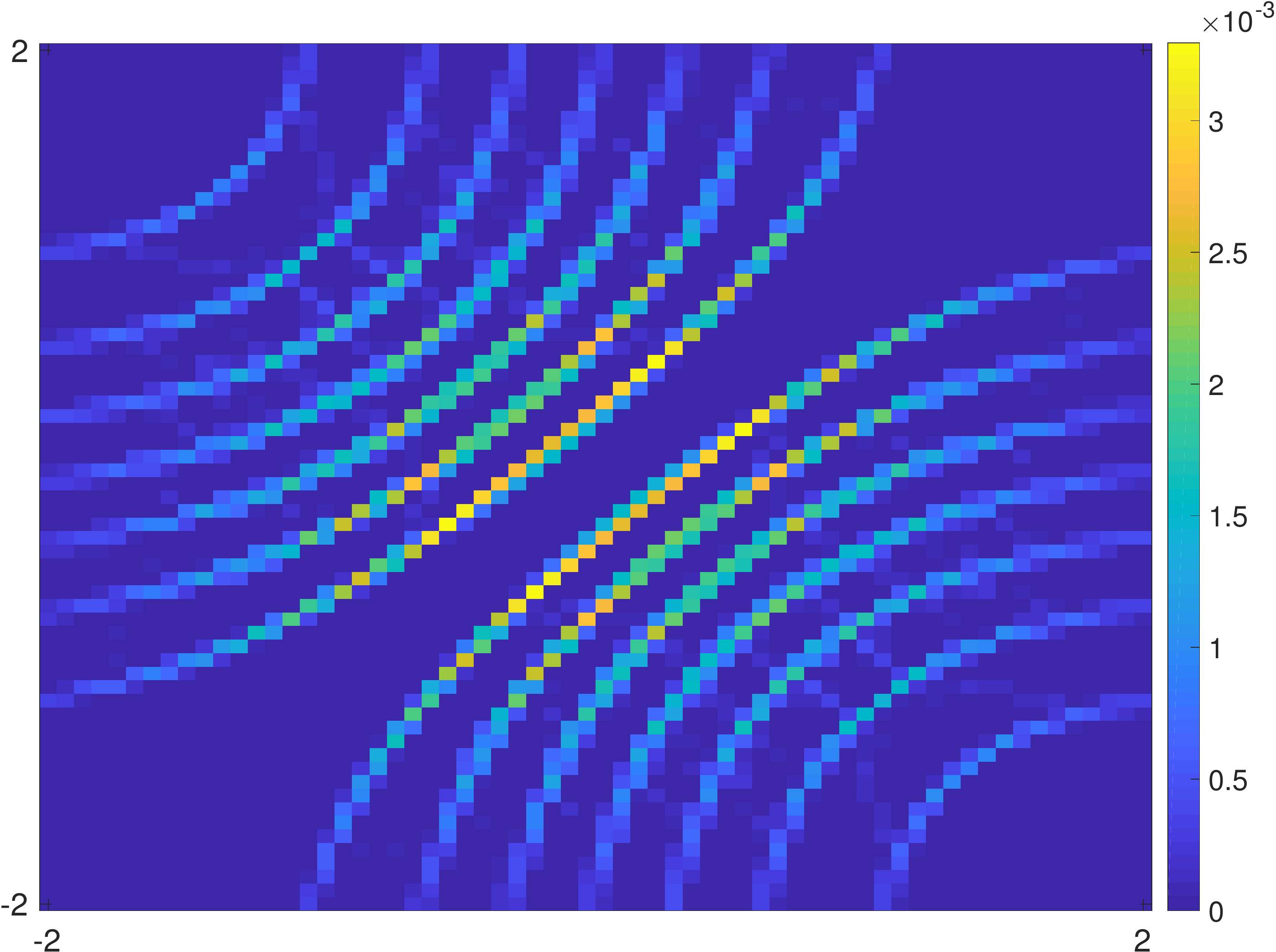}}
	\qquad
	\subfloat[SDP-Coulomb2.]{\includegraphics[width=0.47\columnwidth]{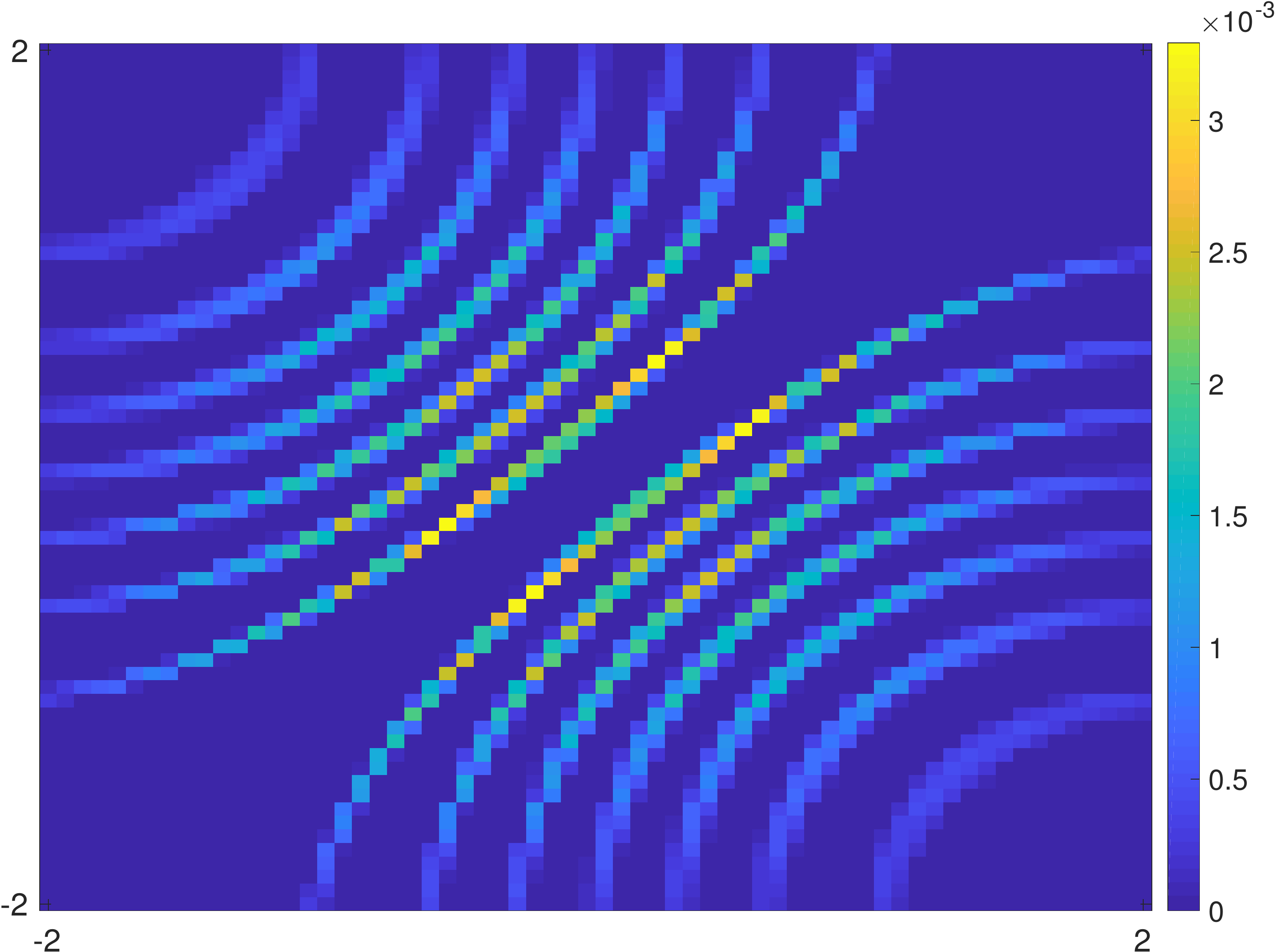}}
	\qquad
	\caption{2-marginal from solving the multimarginal transport problem with the marginal
      $\rho_2(x)$ where $N=8$, $\vert X \vert = 64$, $d=1$. (a): Solution from
      SDP-Coulomb. ${E_\text{gap}}_1=$ 1.8e-03. (b): Solution from SDP-Coulomb2. ${E_\text{gap}}_2=$
      1.5e-03.}\label{figure: gaussian64}
\end{figure}

\begin{figure}[!ht]
	\centering
	\subfloat[SDP-Coulomb.]{\includegraphics[width=0.47\columnwidth]{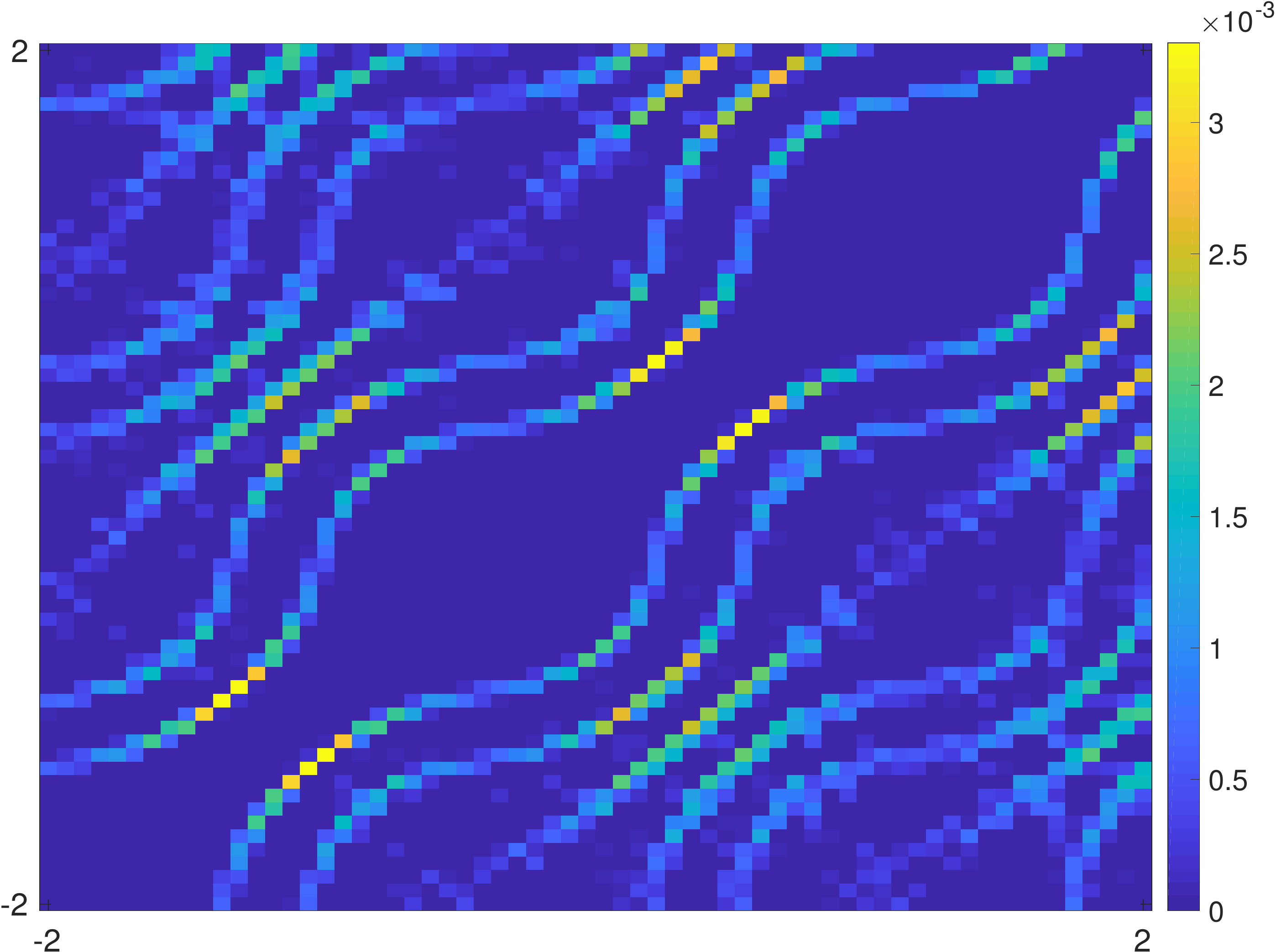}}
	\qquad
	\subfloat[SDP-Coulomb2.]{\includegraphics[width=0.47\columnwidth]{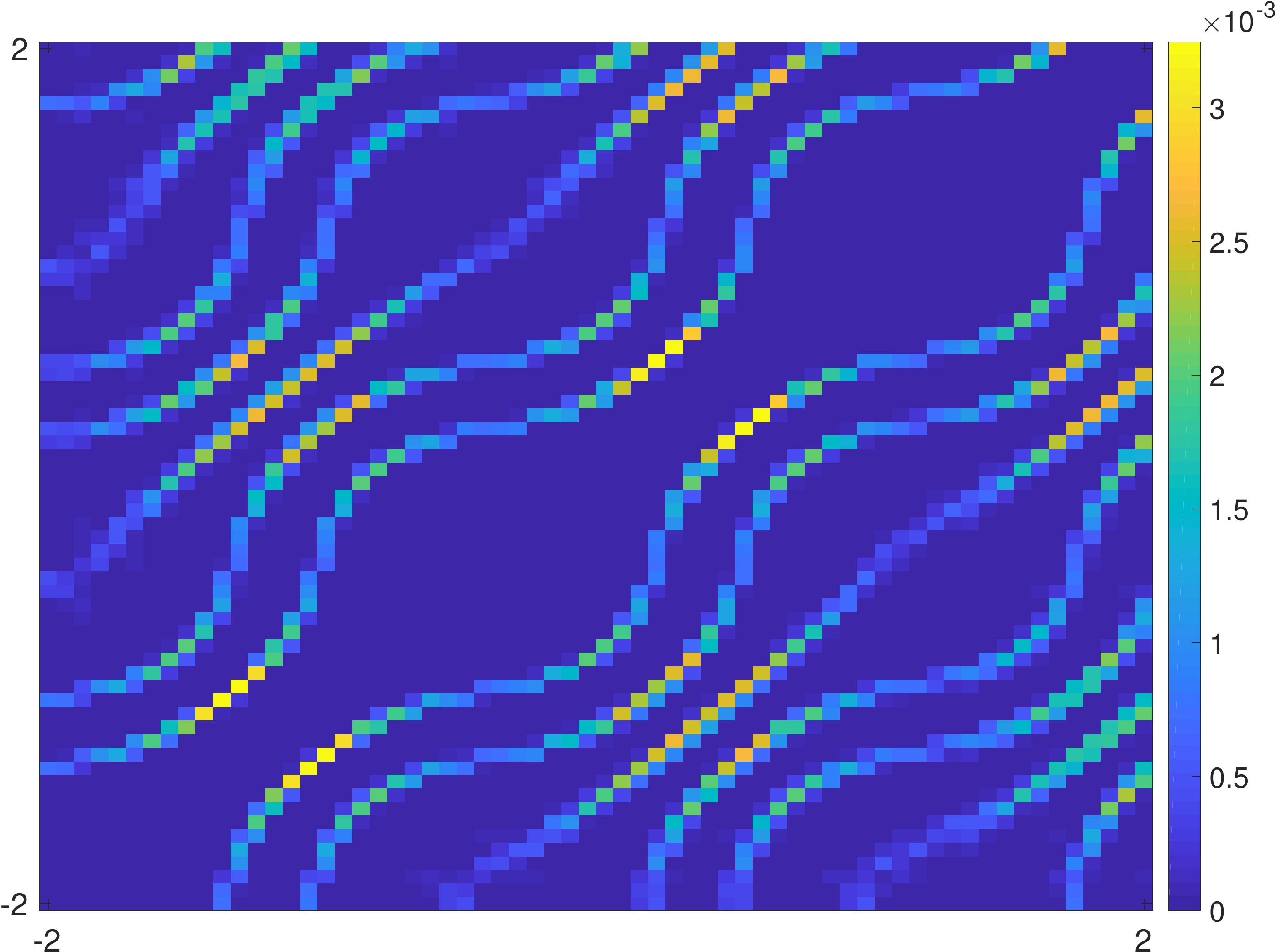}}
	\qquad
	\caption{2-marginal from solving the multimarginal transport problem with the marginal
      $\rho_3(x)$ where $N=8$, $\vert X \vert = 64$, $d=1$. (a): Solution from
      SDP-Coulomb. ${E_\text{gap}}_1=$ 4.2e-02. (b): Solution from SDP-Coulomb2. ${E_\text{gap}}_2=$
      3.9e-02.}\label{figure: sine64}
\end{figure}

\begin{figure}[!ht]
	\centering
	\subfloat[SDP-Coulomb.]{\includegraphics[width=0.47\columnwidth]{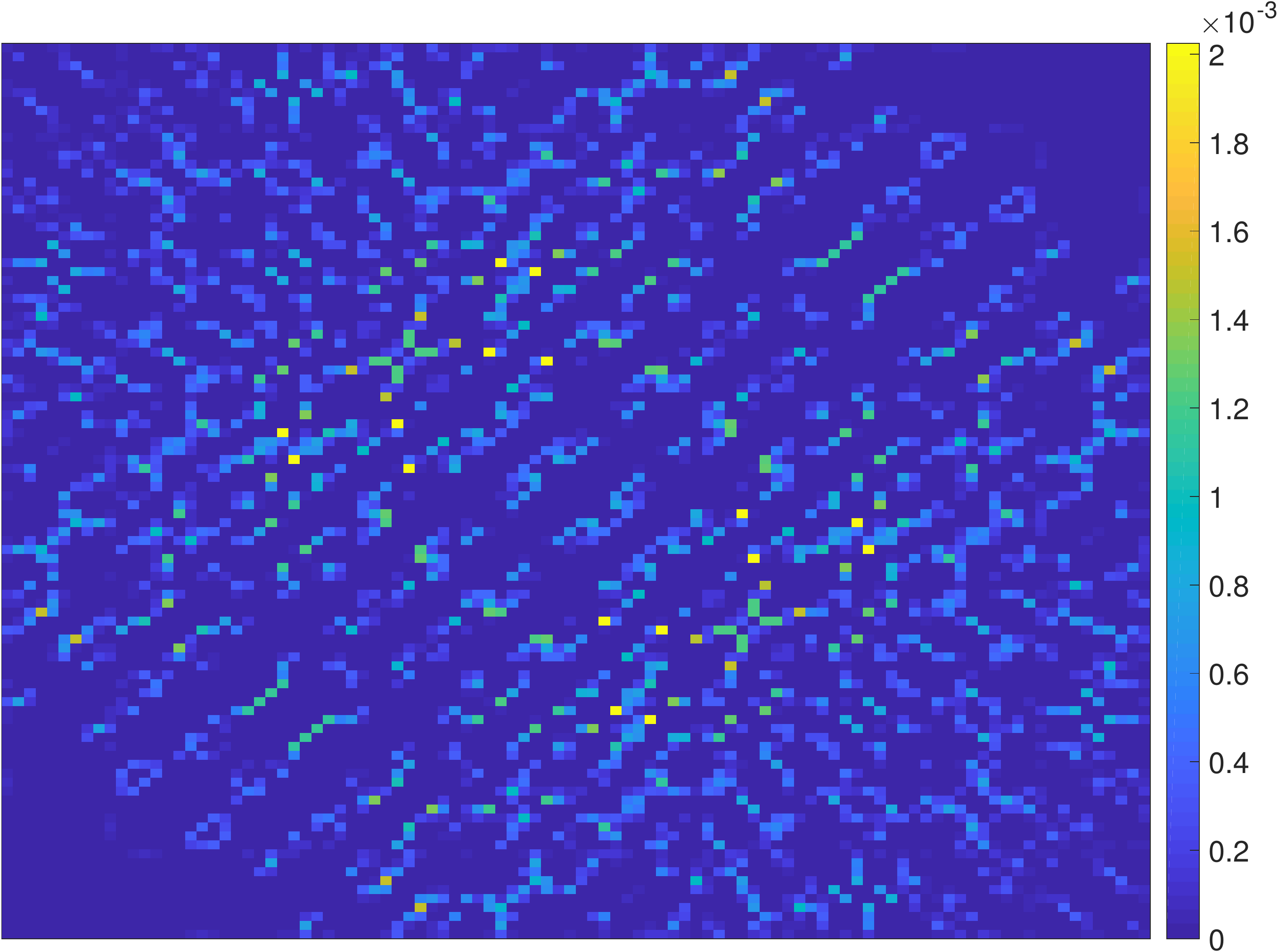}}
	\qquad
	\subfloat[SDP-Coulomb2.]{\includegraphics[width=0.47\columnwidth]{figures/gaussian2d2_100.pdf}}
	\qquad
	\caption{Solution to the multimarginal transport problem with the marginal $\rho_4(x)$ where
      $N=6$, $\vert X \vert = 10^2$, $d=2$. The 2D domain $X$ is vectorized in order to present the
      2-marginal. (a): Solution from SDP-Coulomb. ${E_\text{gap}}_1=$ 3.8e-02. (b): Solution from
      SDP-Coulomb2. ${E_\text{gap}}_2=$ 3.5e-02.}\label{figure: gaussian2d}
\end{figure}

\subsection{Approximating the Kantorovich potential}
As mentioned previously, the dual problem \eqref{vsce dual} can also be used to approximate the
Kantorovich problem \eqref{vsce}. The 1D cases admit semi-analytic solutions for the dual
potential \cite{seidl1999strong}. First, the \emph{comotion function} is defined as
\begin{equation}
  f_i(x)=\begin{cases}
  N_e^{-1}(N_e(x)+i-1), & x\leq N_e^{-1}(N+1-i),  \\
  N_e^{-1}(N_e(x)+i-1-N), & x > N_e^{-1}(N+1-i),
  \end{cases}
\end{equation} 
for $i=1,\ldots,N$, where
\begin{equation}
  N_e := N \int_{-\infty}^x \rho(x) dx.
\end{equation} 
Then the Kantorovich potential $v^*(x)$ is defined via
\begin{equation}
  \label{gt kanto}
  \nabla v^*(x) = -N\sum_{i=1}^N \frac{x- f_i(x)}{\| x-f_i(x) \|_2^3}.
\end{equation}
We compare the dual potential $w^*$ obtained from solving \eqref{vsce dual} to the ground truth
Kantorovich potential \eqref{gt kanto}. We let $\vert X \vert = 200$ and the marginals being
$\rho_1(x),\rho_2(x)$ and $\rho_3(x)$. The error is reported using the metric
\begin{equation}
  \text{Error}_v = \frac{\| v^*-w^*\|_2}{\|v^*\|_2}.
\end{equation}
In these cases, we obtain errors of the order of $10^{-3}$ to $10^{-2}$. The results are presented
in Fig. \ref{figure: kantorovich}.
\begin{figure}[!ht]
	\centering
	\subfloat[$\rho_1(x)$.]{\includegraphics[width=0.4\columnwidth]{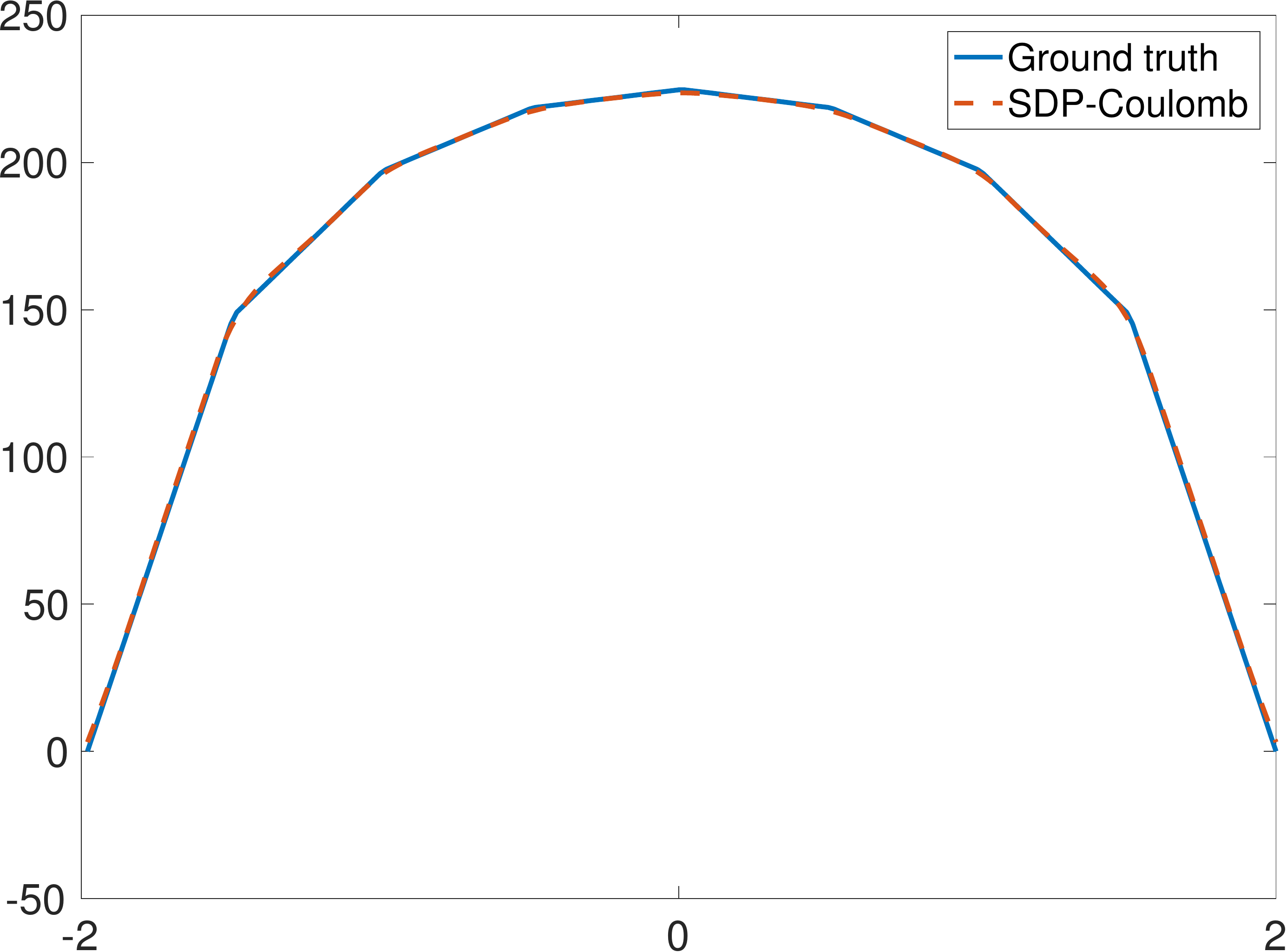}}
	\qquad
	\subfloat[$\rho_2(x)$.]{\includegraphics[width=0.4\columnwidth]{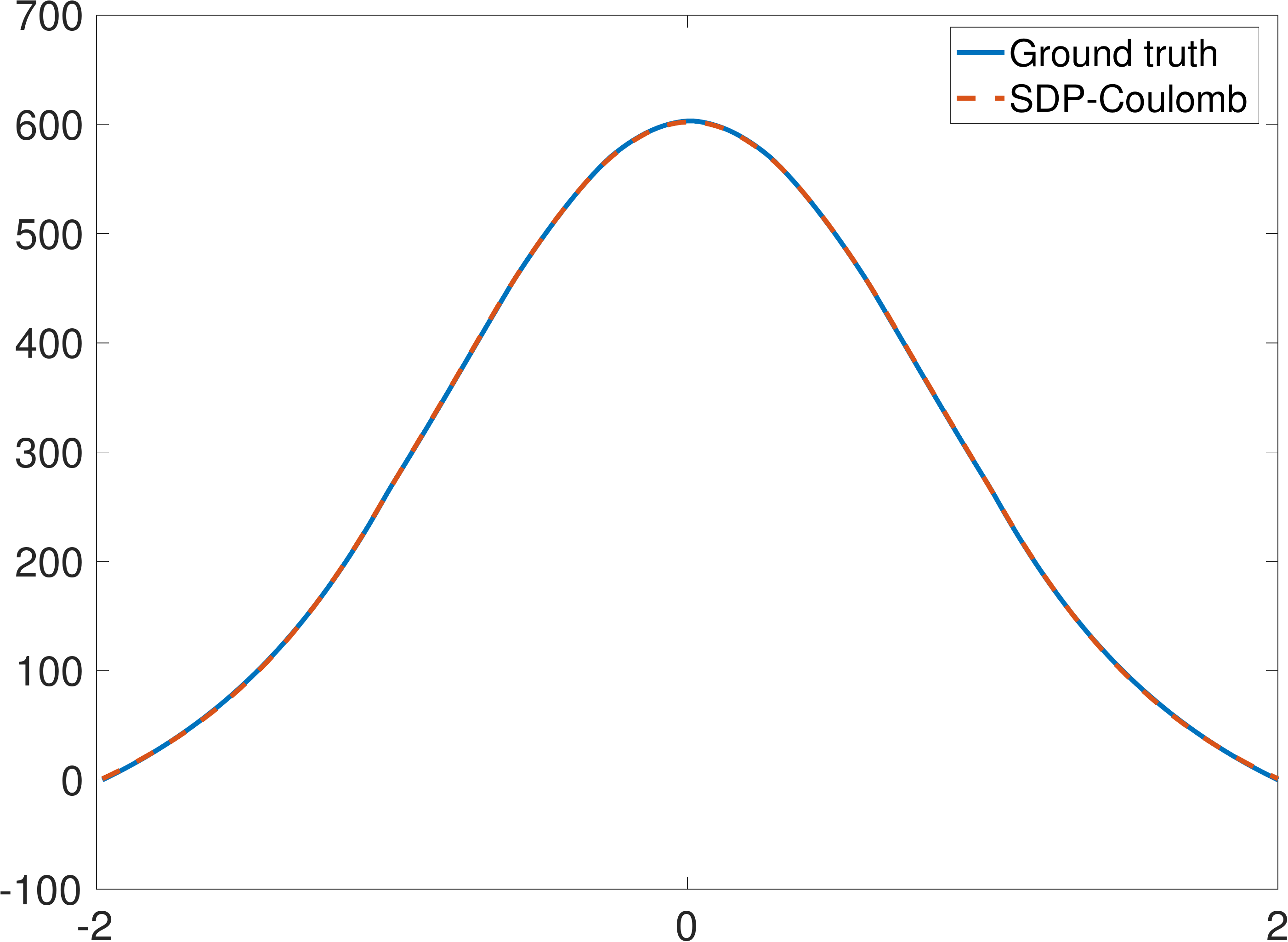}}
	\qquad
	\subfloat[$\rho_3(x)$.]{\includegraphics[width=0.4\columnwidth]{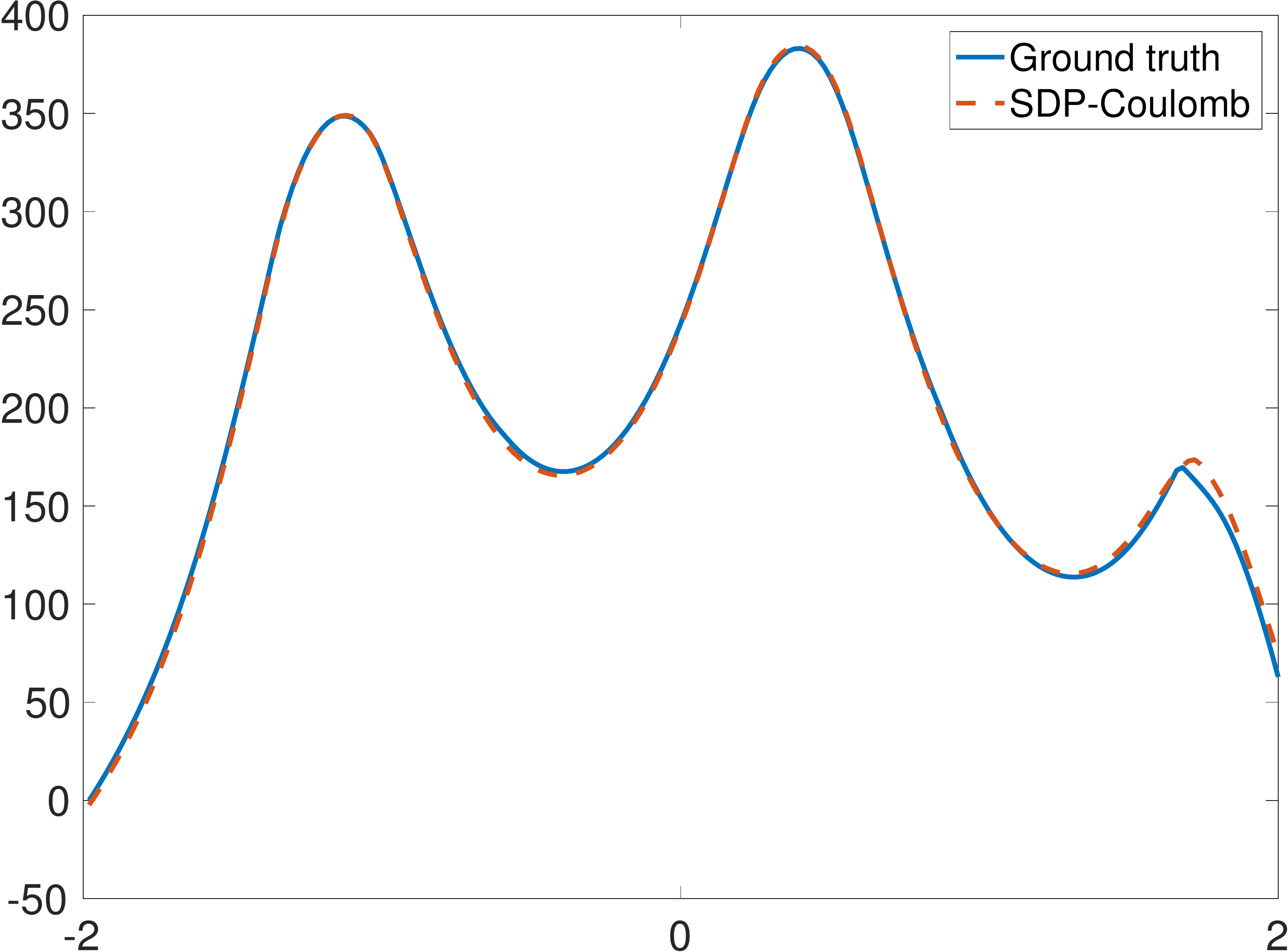}}
	\caption{Solution to the dual problem \eqref{vsce} where $N=8$, $\vert X \vert =
      200$, $d=1$. The ground truth is given by \eqref{gt kanto}, and the approximation is given by
      the solution to the dual problem of SDP-Coulomb \eqref{vsce dual}. (a): With marginal
      $\rho_1(x)$. $\text{Error}_v$=4.5e-03. (b): With marginal
      $\rho_2(x)$. $\text{Error}_v$=1.4e-03.  (c): With marginal
      $\rho_3(x)$. $\text{Error}_v$=1.2e-02. }\label{figure: kantorovich}
\end{figure}

\section{Conclusion}
We propose methods based on convex relaxation for solving the multi-marginal transport type problems
in the context of DFT. By convexly relaxing the domain of 2 and 3-marginals, the resulting convex
optimization problems have computational complexities independent of the number of electrons. For
the numerical simulations presented here, directly applying linear programming or Sinkhorn scaling
based algorithm \cite{benamou2016numerical} to Problem \eqref{symmetric MM} would have led to a
tensor with number of entries between $10^{14}$ to $10^{25}$, for the choice of $N$ and
$\vert{X}\vert$ used here.

Furthermore, a key feature of the proposed methods is that they provide both upper and lower bounds
on the energy.
From an algorithmic point of view, it is crucial to develop faster customized optimizer in order to
address large-scale applications in the future. From a theoretical point of view, it is important to
study theoretically how well SDP-Coulomb and SDP-Coulomb2 approximate Problem \eqref{symmetric MM}.

\section*{Acknowledgments}
The authors thank Prof. Lin Lin for introducing the problem. Y.K. thanks Prof. Emmanuel Cand\`es for
the partial support from a Math+X postdoctoral fellowship. The work of Y.K. and L.Y. is partially supported
by the U.S. Department of Energy, Office of Science, Office of Advanced Scientific Computing
Research, Scientific Discovery through Advanced Computing (SciDAC) program and the National Science
Foundation under award DMS-1818449.

\bibliographystyle{abbrv}
\bibliography{bibref}

\end{document}